\documentclass[11pt]{amsart}
\usepackage[utf8]{inputenc}
\usepackage[pdf]{graphviz}
\usepackage{fourier}
\usepackage[margin=3cm]{geometry}
\usepackage{amssymb}
\usepackage{graphicx}

\newtheorem{theorem}{Theorem}[section]
\newtheorem{proposition}[theorem]{Proposition}
\newtheorem{lemma}[theorem]{Lemma}
\newtheorem{corollary}[theorem]{Corollary}
\newtheorem{question}[theorem]{Question}
\theoremstyle{remark}
\newtheorem{example}[theorem]{Example}
\newtheorem{remark}[theorem]{Remark}

\DeclareMathOperator{\Ap}{Ap}
\usepackage{hyperref}

\title{Apéry sets and the ideal class monoid of a numerical semigroup}
\date{}

\author[Casabella]{Laura Casabella}
\address[Casabella]{Max Planck Institute for Mathematics in the Sciences - 04103 Leipzig, Germany}
\email{laura.casabella@mis.mpg.de}

\author[D'Anna]{Marco D'Anna}
\address[D'Anna]{Dipartimento di Matematica e Informatica, Viale A. Doria, 6 - I 95125 – Catania, Italia}
\email{mdanna@dmi.unict.it}

\author[Garc\'ia-S\'anchez]{Pedro A. Garc\'ia-S\'anchez}
\address[Garc\'ia-S\'anchez]{Departamento de \'Algebra and IMAG, Universidad de Granada, E-18071 Granada, Espa\~na}
\email{pedro@ugr.es} 
\thanks{The third author is supported by the grant number ProyExcel\_00868 (Proyecto de Excelencia de la Junta de Andalucía) and by the Junta de Andaluc\'ia Grant Number FQM--343}

\begin{document}

\maketitle
\begin{abstract}
The aim of this article is to study the ideal class monoid $\mathcal{C}\ell(S)$ of a numerical semigroup $S$ introduced by V. Barucci and F. Khouja. We prove new bounds on the cardinality of $\mathcal{C}\ell(S)$. We observe that $\mathcal{C}\ell(S)$ is isomorphic to the monoid of ideals of $S$ whose smallest element is 0, which helps to relate $\mathcal{C}\ell(S)$ to the Apéry sets and the Kunz coordinates of $S$. We study some combinatorial and algebraic properties of $\mathcal{C}\ell(S)$, including the reduction number of ideals, and the Hasse diagrams of $\mathcal{C}\ell(S)$ with respect to inclusion and addition. From these diagrams we can recover some notable invariants of the semigroup. Lastly, we prove some results about irreducible elements, atoms, quarks and primes of $(\mathcal{C}\ell(S),+)$. Idempotent ideals coincide with over-semigroups and idempotent quarks correspond to unitary extensions of the semigroup. We show that a numerical semigroup is irreducible if and only if $\mathcal{C}\ell(S)$ has at most two quarks.
\end{abstract}

\section{Introduction}
The ideal class group of a Dedekind domain is a classical mathematical object that has been extensively studied, and has proven to be a useful tool to retrieve information about the underlying domain. This concept can be generalized to the ideal class monoid of any integral domain, which is defined as the set of fractional ideals modulo principal ideals, endowed with multiplication of classes.

The ideal class monoid $\mathcal{C}\ell(S)$ of a numerical semigroup $S$ is defined analogously by considering the set of ideals modulo principal ideals of $S$, endowed with addition. This object is strictly related to the ideal class monoid of the associated semigroup ring $\mathbb{K}\llbracket S\rrbracket$ via the valuation map $v: \mathcal{C}\ell(\mathbb{K}\llbracket S\rrbracket) \to \mathcal{C}\ell(S)$ \cite{laura}. The latter object is not very well-understood. However, by working on $\mathcal{C}\ell(S)$ it might be possible to interpret $\mathcal{C}\ell(\mathbb{K}\llbracket S\rrbracket)$ from a combinatorial (and hence easier) point of view. 

The aim of this article is to extend the study of $\mathcal{C}\ell(S)$ which was initiated by V. Barucci and F. Khouja in \cite{b-k}. We prove new estimations of the cardinality of $\mathcal{C}\ell(S)$. The most relevant results in this direction are Propositions~\ref{prop:lo-up-bound-clS-genus-type} and \ref{improved}, which ensure that if $g$ is the genus of $S$, $t$ its type and $m$ its multiplicity, then 
\[2^{m-1}+g-m+1 \le |\mathcal{C}\ell(S)| \le 2^g-2^{g-t}+1.\]
Moreover, the upper bound is reached if and only if either $S= \{0,m, \to \}$ or $S= \langle m, m+1, \dots, 2m-2 \rangle$.

We then observe that $\mathcal{C}\ell(S)$ is isomorphic to the monoid $\mathcal{I}_0(S)$ of ideals of $S$ whose smallest element is 0. This will allow us to relate $\mathcal{C}\ell(S)$ to the Apéry sets and the Kunz coordinates of $S$ (see Theorem~\ref{kunz}). Thanks to this new setting, we will be able to prove further results on the cardinality of $\mathcal{C}\ell(S)$ (Corollary~\ref{cor:bound-ideals-kunz}) and to retrieve well-known results (Proposition~\ref{known}). In particular, if $m$ is the multiplicity of $S$ and $(k_1,\dots,k_{m-1})$ are the Kunz coordinates of $S$, then  
 \[|\mathcal{C}\ell(S)|\le \prod_{i=1}^{m-1} (k_i+1).\]
Equality holds if and only if $S=\langle m\rangle \cup (c+\mathbb{N})$, with $c$ a positive integer greater than $m$, or, equivalently,
\begin{enumerate}
    \item $k_1\ge \dots \ge k_{m-1}$, and
    \item $k_1-k_{m-1}\le 1$.
\end{enumerate}

Next, we study some combinatorial and algebraic properties of $\mathcal{C}\ell(S)$. Proposition~\ref{prop:canonical-apery} provides a description of the canonical ideal of $S$ in terms of its Apéry set. Then we study the reduction number of an ideal $I$, which is defined, for $I \in \mathcal{I}_0$, as the smallest $r$ such that $(r+1)I=rI$. For an ideal $I$ in $\mathcal{I}_0$ generated by 0 and a single gap of $S$, we compute the reduction number of $I$ in Proposition~\ref{prop:basic-reduction}. In Proposition~\ref{pr:red-gaps-smaller-mult} we give an upper bound of the reduction number of ideals generated by gaps that are smaller than the multiplicity of the semigroup. 

We show that $\mathcal{I}_0(S)$ has a unique maximal element with respect to inclusion, $\mathbb{N}$, and a minimal element, $S$. The number of maximal elements strictly contained in $\mathbb{N}$ equals the type of the semigroup, while the number of minimal elements strictly containing $S$ is precisely the mulitiplicity minus one. The maximal strictly ascending chain of ideals in $\mathcal{I}_0(S)$ has length equal to the genus of the semigroup plus one. Thus, from the Hasse diagram (with respect to inclusion) of $\mathcal{I}_0(S)$ we recover the multiplicity, type and genus of $S$. In Remark~\ref{width}, we also give a lower bound for the width of this Hasse diagram. 

In Section~\ref{quarks}, we focus on the study of irreducible elements, atoms, primes and quarks of the ideal class monoid of a numerical semigroup. Minimal non-zero ideals are always quarks, and for every gap $g$, the (class of the) ideal $\{0,g\}+S$ is an irreducible element in $\mathcal{C}\ell(S)$. Also, the set of unitary extensions of $S$ corresponds with the set of idempotent quarks of its ideal class monoid. We show that the semigroup $S$ is irreducible (it cannot be expressed as the intersection of two numerical semigroups properly containing it) if and only if its ideal class monoid has at most two quarks. 

The last section is devoted to several open problems that may serve as a motivation to continue studying the ideal class monoid of a numerical semigroup. 

Throughout this paper we present a series of examples meant to illustrate the results proven. For the development of most of these examples we used the \texttt{GAP} \cite{GAP4} package \texttt{numericalsgps} \cite{numericalsgps} (in fact, some of our results were stated after analysing a series of computer experiments). The functions used in this manuscript, together with a small tutorial, can be found at 

\centerline{
\url{https://github.com/numerical-semigroups/ideal-class-monoid}.}

\section{Recap on numerical semigroups and ideals}

Let $\mathbb{N}$ denote the set of non-negative integers. A numerical semigroup $S$ is a submonoid of $(\mathbb{N},+)$ with finite complement in $\mathbb{N}$. The set $\mathbb{N} \setminus S$ is known as the set of gaps of $S$, denoted $\operatorname{G}(S)$, and its cardinality is the genus of $S$, $\operatorname{g}(S)$. Given a subset $A$ of $\mathbb{N}$, the submonoid generated by $A$ is $\langle A\rangle = \{ a_1+\dots+ a_t : t\in \mathbb{N}, a_1,\dots,a_t\in A\}$. If $A\subseteq S$ is such that $\langle A\rangle=S$, then we say that $A$ is a generating set of $S$. Every numerical semigroup admits a unique minimal generating set (whose elements we refer to as minimal generators), which is $S^*\setminus (S^*+S^*)$, where $S^*=S\setminus\{0\}$, and its cardinality is known as the embedding dimension of $S$, denoted $\operatorname{e}(S)$ (see for instance \cite[Chaper~1]{ns}). The smallest positive integer in $S$ is called the multiplicity of $S$, $\operatorname{m}(S)$. Clearly, two minimal generators of $S$ cannot be congruent modulo the multiplicity of $S$ and, consequently, $\operatorname{e}(S)\le \operatorname{m}(S)$.

The largest integer not belonging to a numerical semigroup $S$ (this integer exists as we are assuming $\operatorname{G}(S)$ to have finitely many elements) is known as the Frobenius number of $S$ and will be denoted by $\operatorname{F}(S)$. In particular, $\operatorname{F}(S)+1+\mathbb{N}\subseteq S$. The integer $\operatorname{F}(S)+1$ is the conductor of $S$, $\operatorname{c}(S)$. 

A numerical semigroup $S$ induces the following ordering on $\mathbb{Z}$: $a\le_S b$ if $b-a\in S$. The set of maximal elements in $\operatorname{G}(S)$ with respect to $\le_S$ is denoted by $\operatorname{PF}(S)$, and its elements are the pseudo-Frobenius numbers of $S$. Observe that by the maximality of these elements, if $f\in \operatorname{PF}(S)$, then for every non-zero $s\in S$, we have that $f+s\in S$. The cardinality of $\operatorname{PF}(S)$ is known as the type of $S$, and will be denoted by $\operatorname{t}(S)$. 

Recall that a gap $f$ of a numerical semigroup $S$ is a special gap if $S\cup\{f\}$ is a numerical semigroup. We denote the set of special gaps of $S$ by $\operatorname{SG}(S)$. It is well known (see for instance \cite[Section~3.3]{ns}) that
\[
\operatorname{SG}(S)=\{f\in \operatorname{PF}(S) : 2f\in S\}.
\]
In particular, the cardinality of $\operatorname{SG}(S)$ corresponds with the number of unitary extensions of a numerical semigroup (numerical semigroups $T$ containing $S$ such that $|T\setminus S|=1$). 

Given a numerical semigroup $S$ and a non-zero element $n\in S$, the Apéry set of $n$ in $S$ is the set 
\[
\operatorname{Ap}(S,n)=\{ s\in S : s-n\not\in S\}.
\]
This set contains $n$ elements, one per congruence class modulo $n$. In fact, if $w_i$ is the least element in $S$ congruent to $i$ modulo $n$, $i\in \{0,\ldots,n-1\}$, then $\operatorname{Ap}(S,n)=\{w_0,w_1,\dots,w_{n-1}\}$, and clearly, $w_0=0$. Recall that (see for instance \cite[Proposition~2.20]{ns}) for a numerical semigroup $S$ with multiplicity $m$, we have that 
\begin{equation}\label{eq:pF-max}
\operatorname{PF}(S)=\{f\in \mathbb{Z}\setminus S: f+(S\setminus\{0\})\subseteq S\}=-m+\operatorname{Maximals}_{\le_S}(\Ap(S,m)).
\end{equation}

A numerical semigroup $S$ is symmetric if for every integer $x\not\in S$, $\operatorname{F}(S)-x\in S$. This is equivalent to $\operatorname{g}(S)=(\operatorname{F}(S)+1)/2$, or to the fact that $\operatorname{F}(S)$ is odd and $S$ is maximal (with respect to set inclusion) in the set of numerical semigroups not containing $\operatorname{F}(S)$. If $\operatorname{F}(S)$ is even, then $S$ is said to be pseudo-symmetric if for every integer $x\not\in S$, $x\neq \operatorname{F}(S)/2$, we have that $\operatorname{F}(S)-x\in S$. This is equivalent to $\operatorname{g}(S)=(\operatorname{F}(S)+2)/2$, or to the fact that $S$ is maximal in the set of numerical semigroups not containing $\operatorname{F}(S)$. A numerical semigroup $S$ is irreducible if it cannot be expressed as the intersection of two numerical semigroups properly containing it, or equivalently, it is maximal in the set of numerical semigroups not containing $\operatorname{F}(S)$. Thus a numerical semigroup $S$ is irreducible if and only if it is either symmetric (and its Frobenius number is odd) or pseudo-symmetric (and its Frobenius number is even). These equivalences and other characterisations of irreducible numerical semigroups can be found in \cite[Chapter~3]{ns}.

Let $S$ be a numerical semigroup. We say that $E\subseteq \mathbb{Z}$ is an ideal of $S$ if $S+E\subseteq E$ and there exists $s\in S$ such that $s+E\subseteq S$. This last condition implies that there exists $\min(E)$, which is usually known as the multiplicity of $E$. Notice that if $E$ is an ideal, so is $-\min(E)+E$. 

An ideal $E$ of $S$ is said to be integral if $E\subseteq S$. The set $S\setminus\{0\}$ is known as the maximal ideal of $S$.

Given a set of integers, $\{x_1,\dots,x_r\}$, the set $\{x_1,\dots,x_r\}+S=\bigcup_{i=1}^r (x_i+S)$ is an ideal of $S$, known as the ideal generated by $\{x_1,\dots,x_r\}$. When $r=1$, we write $x_1+S$ instead of $\{x_1\}+S$, and we say that $x_1+S$ is a principal ideal.

Let $E$ be an ideal of a numerical semigroup $S$. Let $m$ be the multiplicity of $S$ and set $x_1=\min(E)$. Then $x_1+S\subseteq E$. If $x_1+S=E$, then $E$ is a principal ideal. Otherwise, take $x_2=\min(E\setminus(x_1+S))$. Clearly, $\{x_1,x_2\}+S\subseteq E$. Moreover, $x_1$ and $x_2$ are not congruent modulo $m$, since $x_2\not\in x_1+S$. This process must stop after a finite number of steps, since the $x_i$s obtained are not congruent modulo $m$. Thus, $E=\{x_1,\dots,x_r\}+S$ for some $\{x_1,\dots,x_r\}\subseteq \mathbb{Z}$, $r\le m$. This set is not only a generating set of $E$, but a minimal generating set of $E$ in the sense that none of its proper subsets $X$ verifies that $X+S=E$. The cardinality of the minimal generating set of $E$ is known as the embedding dimension of $E$ and it is denoted as $\nu(E)$. Clearly, $\nu(E)\le m$.

Given two ideals $I$ and $J$ of $S$, the following sets are also ideals of $S$:
\begin{itemize}
    \item $I\cap J$,
    \item $I\cup J$,
    \item $I+J = \{i+j : i\in I, j\in J\}$.
\end{itemize}

For an ideal $E$ of $S$, we define $\operatorname{F}(E)=\max(\mathbb{Z}\setminus E)$.

\section{The ideal class monoid of a numerical semigroup}

Let $S$ be a numerical semigroup. The set of all ideals of $S$, 
\[
\mathcal{I}(S)=\{E\subset \mathbb{Z} : E \text{ is an ideal of } S\},
\]
is a monoid with respect to addition, since for every ideals $I$ and $J$ of $S$, $I+J$ is also an ideal of $S$, and $I+S=I$ for every $I$ (and so $S$ is the identity element of this monoid). 

\begin{remark}\label{rem:non-unit-cancellative}
    The monoid $(\mathcal{I}(S),+)$ is not cancellative ($I+J=I+K$ for $I,J,K\in \mathcal{I}(S)$ implies $J=K$). As a matter of fact, it is not even unit-cancellative ($I+J=J$ implies $I$ is a unit). To see this, let $F$ be the Frobenius number of $S$, and set $I=\{0,F\}+S$. Then it is easy to see that $I+I=I$. If there exists $J\in \mathcal{I}(S)$ such that $I+J=S$, then $J=0+J\subseteq S$. Also $F+J\subseteq S$, which forces $0\not\in J$. As $0\in S=J+I$, $0=j+i$ for some $j\in J$ and $i\in I$. But this is impossible, since $j>0$ ($J\subseteq S\setminus\{0\}$) and $i\ge 0$.

    Observe also that if we consider only integral ideals of $S$ (ideals contained in $S$), then the resulting monoid is unit-cancellative. Let $I$ and $J$ be two ideals of $S$ with $I\subseteq S$ and $J\subseteq S$. Set $i=\min(I)$ and $j=\min(J)$. If $I+J=I$, then $i=x+y$ for some $x\in I$ and $j\in J$. Thus $i\ge i+j\ge i$, which forces $j=0$. But then $0\in J$, and $J+S=S$ yields $S\subseteq J$, whence $S=J$.
\end{remark}

On $\mathcal{I}(S)$ we define the following equivalence relation: $I\sim J$ if there exists $z\in \mathbb{Z}$ such that $I=z+J$. Clearly, if $I\sim J$ and $I'\sim J'$, then $I+I'\sim J+J'$, and consequently $\sim$ is a congruence. This makes 
\[\mathcal{C}\ell(S)=\mathcal{I}(S)/\sim\] a monoid, which is known as the ideal class monoid of $S$.

\begin{remark}\label{rem:class-semigroup}
    The ideal class monoid of $S$ does not have to be confused with the class semigroup of $S$, \cite[Section~2.8]{g-hk}. As $S$ lives in the free monoid $\mathbb{N}$, we can define the class semigroup, $\mathcal{C}(S,\mathbb{N})$, as the quotient of $\mathbb{N}$ modulo the relation $x\sim y$ if $(-x+S)\cap \mathbb{N}=(-y+S)\cap\mathbb{N}$. Notice that if $c$ is the conductor of $S$ and $x\ge c$, then $(-x+S)\cap \mathbb{N}=\mathbb{N}$. Thus the class of $c$ under the relation $\sim$ is equal to $c+\mathbb{N}$. If $x$ is a non-negative integer less than $c$, then $\max(\mathbb{N}\setminus(-x+S))=F-x$, with $F$ the Frobenius number of $S$. So if $(-x+S)\cap \mathbb{N}=(-y+S)\cap \mathbb{N}$, then their complements in $\mathbb{N}$ are equal, and so $F-x=F-y$, yielding $x=y$. Thus 
    \[\mathcal{C}(S,\mathbb{N})=\{\{0\},\{1\},\dots,\{c-1\},c+\mathbb{N}\}.\]
\end{remark}

Set 
\[
\mathcal{I}_0(S)=\{ E \subset \mathbb{N} : E+S\subseteq E, \min(E)=0\}.
\]
Notice that $(\mathcal{I}_0(S), +)$ is also a monoid (with identity element $S$), and it is isomorphic to $\mathcal{C}\ell(S)$; the isomorphism is $E\mapsto [E]$, since every class $[I]$ of $\mathcal{C}\ell(S)$ contains $-\min(I)+I$, which is in $\mathcal{I}_0(S)$. 

Observe also that, for all $I$ and $J$ in $\mathcal{I}_0(S)$,
\[
I\cup J\subseteq I+J.
\]

\begin{proposition}[\protect{\cite{b-k}}]\label{prop:reduced}
    Let $S$ be a numerical semigroup. The only invertible element of $\mathcal{C}\ell(S)$ is $[S]$. In particular, $\mathcal{C}\ell(S)$ is a group if and only if $S=\mathbb{N}$.
\end{proposition}
\begin{proof}
    Let $I$ be an ideal of $\mathcal{I}_0(S)$ such that there exists $J\in \mathcal{I}_0(S)$ with $I+J=S$. We have $I\subseteq I+J=S=0+S\subseteq I$, which forces $I=S$.
\end{proof}

Clearly, for every $I\in \mathcal{I}_0(S)$, since $0\in I$, there is no other minimal generator of $I$ contained in $S$. Hence, there exists $X\subseteq \mathbb{N}\setminus S$ such that $I=(\{0\}\cup X)+S$. In particular, 
\begin{equation}\label{eq:ubound-cS-2genus}
    |\mathcal{C}\ell(S)|\le 2^\mathrm{g},
\end{equation}
which was already shown in \cite{b-k} ($|\cdot|$ denotes cardinality). It is also clear that for every $g\in \mathbb{N}\setminus S$, the ideal $\{0,g\}+S \in \mathcal{I}_0(S)$, and so
\begin{equation}\label{eq:lbound-cS-genus}   
g+1\le |\mathcal{C}\ell(S)|
\end{equation}
(the plus one corresponds to the ideal $S$).

We now recall the notion of Hasse diagram of a poset. Let $(P, \leq) $ be a finite poset, and let $a,b$ be two elements of $P$. We say that $b$ covers $a$ if $a < b$ and there is no $ c \in P$  such that $a < c < b$. The Hasse diagram of $P$ is the graph whose set of vertices is $P$ and whose edges are the covering relations. An antichain of $P$ is a set of pairwise incomparable elements of $P$. The width of $P$ is the biggest cardinality of an antichain of $P$.

Let $S$ be a numerical semigroup and consider the Hasse diagram of the poset $(G(S), \leq_S)$. It has the following properties: 

\begin{itemize}
    \item $\operatorname{Minimals}_{\le_S}(\operatorname{G}(S))=\{1,\dots,\operatorname{m}(S)-1\}$, and so $|\operatorname{Minimals}_{\le_S}(\operatorname{G}(S))|+1=\operatorname{m}(S)$.
    \item $\operatorname{Maximals}_{\le_S}(\operatorname{G}(S))=\operatorname{PF}(S)$, and consequently $|\operatorname{Maximals}_{\le_S}(\operatorname{G}(S))|=\operatorname{t}(S)$.
    \item The width of $(\operatorname{G}(S),\le_S)$ is $\operatorname{m}(S)-1$ because $\operatorname{Minimals}_{\le_S}(\operatorname{G}(S))=\{1,\dots,\operatorname{m}(S)-1\}$ is a maximal antichain (a set with at least $\operatorname{m}(S)$ gaps will contain at least two elements congruent modulo $\operatorname{m}(S)$ and thus they will be comparable via $\le_S$). According to Dilworth's Theorem \cite{d}, there exists a partition of $\operatorname{G}(S)$ into $\operatorname{m}(S)-1$ chains. This partition is easily achieved with the chains $C_i=\{i,i+\operatorname{m}(S),\dots, i+(k_i-1)\operatorname{m}(S)\}$, where $k_i=\min \{ k\in \mathbb{N} : i+km\in S\}$. 
\end{itemize}

The Hasse diagram of $(G(S), \leq_S)$ is related to the study of the ideal class monoid of $S$: if $I\in \mathcal{I}_0(S)$, then its minimal generating set is of the form $\{0\}\cup X$, with $X$ a set of gaps incomparable with respect to $\le_S$. Thus, we obtain the following.

\begin{proposition}\label{prop:card-clS-antichain}
    Let $S$ be a numerical semigroup. The cardinality of $\mathcal{C}\ell(S)$ equals the number of antichains of gaps of $S$ with respect to $\le_S$. 
\end{proposition}

Having in mind that the pseudo-Frobenius numbers are incomparable gaps with respect to the ordering induced by the numerical semigroup, we can find a lower bound for the cardinality of the ideal class monoid.

\begin{proposition}\label{prop:lo-up-bound-clS-genus-type}
    Let $S$ be a numerical semigroup with genus $g$ and type $t$, $S\neq \mathbb{N}$. Then
    \[
    2^t\le |\mathcal{C}\ell(S)| \le 2^g-2^{g-t}+1.
    \]
\end{proposition}
\begin{proof}
Let $f_1,\dots,f_r$ be pseudo-Frobenius numbers of $S$. Then, by \eqref{eq:pF-max}, $\{f_1,\dots,f_r\}$ is an antichain of gaps of $S$ with respect to $\le_S$, and thus $\{0,f_1,\dots,f_r\}$ is a minimal generating set of the ideal $\{0,f_1,\dots,f_r\}+S\in \mathcal{I}_0(S)$. This in particular means that $2^t \le |\mathcal{C}\ell(S)|$. 

For the other inequality, let $\mathcal{P}(\mathbb{N}\setminus S)$ be the set of subsets of gaps of $S$. Consider the following map: $\mathcal{G}:\mathcal{I}_0(S)\to \mathcal{P}(\mathbb{N}\setminus S)$, $E\mapsto E\setminus S$. Clearly, $E=S\cup(E\setminus S)$ and $E\setminus S\subseteq \mathbb{N}\setminus S$, and so $\mathcal{G}$ is injective. Let us prove that if $E\neq S$, then $E\setminus S$ contains at least a pseudo-Frobenius number of $S$. Let $x$ be in $E\setminus S$. Since $\operatorname{PF}(S)=\operatorname{Maximals}_{\le_S}(\mathbb{Z}\setminus S)$, there exists $f\in \operatorname{FP}(S)$ such that $x\le_s f$. Hence there exists $s\in S$ such that $x+s=f$, and consequently $f\in E\setminus S$. This means that the image of $\mathcal{G}$ is included in the set of subsets of gaps of $S$ with at least one pseudo-Frobenius number, and the cardinality of this set is $2^{g-t}(2^t-1)=2^g-2^{g-t}$.
\end{proof}

\begin{remark}
  Let us see that the upper bound in Proposition~\ref{prop:lo-up-bound-clS-genus-type} for $m\ge 3$ is attained if and only if either $S=\{0,m,\to\}$ (here $\to$ means that every integer greater than $m$ is in the semigroup) or $S=\langle m,m+1,\dots,2m-2\rangle$. 

For the sufficiency, notice that $\operatorname{G}(S)$ is either $\{1,\dots,m-1\}$ or $\{1,\dots,m-1,2m-1\}$. In the first case, for every subset $X$ of $\{1,\dots,m-1\}$, we have that $\{0\}\cup X\cup S$ is an ideal of $S$, and so $|\mathcal{I}_0(S)|\ge 2^{m-1}$. Also, the type of $\{0,m,\to\}$ is $m-1$, which equals the genus. So $2^g-2^{g-t}+1=2^{m-1}-2^0+1=2^{m-1}$ and thus the equality holds. For the second case, we have that $\operatorname{G}(S)=\{1,\dots,m-1,2m-1\}$, $\operatorname{PF}(S)=\{2m-1\}$, and $\operatorname{g}(S)=m$. For every subset $X$ of $\{1,\dots,m-1\}$, the set $X\cup\{0,2m-1\}\cup S$ is an ideal of $S$. This makes $2^{m-1}$ elements in $\mathcal{I}_0(S)$ to which we must add $S$ itself, and so we have at least $2^{m-1}+1$ ideals in $\mathcal{I}_0(S)$, and our upper bound in this case is $2^g-2^{g-t}+1=2^{m}-2^{m-1}+1=2^{m-1}+1$; consequently we get an equality once more.

Now let us focus on the necessity. From the proof of Proposition~\ref{prop:lo-up-bound-clS-genus-type}, it follows that the equality $|\mathcal{C}\ell(S)| = 2^g-2^{g-t}+1$ implies that for every subset $X\subseteq \operatorname{G}(S)$, containing at least one pseudo-Frobenius number, the set $X\cup S$ is an ideal of $S$. In particular, $\{i,\operatorname{F}(S)\}\cup S$ is an ideal for every $i\in\{1,\dots,m-1\}$. Thus, $(\{i,\operatorname{F}(S)\}\cup S)+S\subseteq \{i,\operatorname{F}(S)\}\cup S$, which forces $\{i,\operatorname{F}(S)\}+S\subseteq \{i,\operatorname{F}(S)\}\cup S$. Hence $i+m\in \{\operatorname{F}(S)\}\cup S$ for all $i\in \{1,\dots,m-1\}$. Notice that $1+m=\operatorname{F}(S)$ implies that $S=\{0,m,m+2,\to\}$, but then $\operatorname{PF}(S)=\{2,\dots,m-1,m+1\}$, and $\{1,2\}\cup S$ is not an ideal, since $1+m\not\in \{1,2\}\cup S$, contradicting that for any set of gaps $X$ containing at least a pseudo-Frobenius number, the set $X\cup S$ is an ideal of $S$. Hence $m+1\in S$. Now assume that after applying this argument we have that $m,m+1,\dots,m+i-1\in S$ and $m+i=\operatorname{F}(S)$, $i<m-1$. Then $S=\{0,m,m+1,\dots,m+i-1,m+i+1,\to\}$ and $\operatorname{PF}(S)=\{i+1,\dots,m-1,m+i\}$. Notice that $\{1,i+1\}\cup S$ is not an ideal of $S$, because $1+m+i-1=m+i\not\in \{1,i+1\}\cup S$, contradicting again that for any set of gaps $X$ of $S$ containing at least a pseudo-Frobenius number, the set $X\cup S$ is an ideal of $S$. Hence for all $i<m-1$, we obtain $m+i\in S$. Thus it remains to see what happens with $2m-1$ ($i=m+1$). If $2m-1\in S$, then $S=\{0,m,\to\}$, while if $2m-1=\operatorname{F}(S)$, we get $S=\langle m,m+1,\dots,2m-2\rangle$. 

Now notice that all numerical semigroups with $m = 2$ are of the form $S=\langle 2,b \rangle$ where $b$ is an odd integer greater or equal than three. In this case, the Hasse diagram of $\operatorname{G}(S)$ is a chain of length $g$ and thus by Proposition~\ref{prop:card-clS-antichain} the cardinality of $\mathcal{I}_0(S)$ is $g+1$. As $S$ is symmetric, the type of $S$ is one. Observe that $2^g-2^{g-1}+1=g+1$ holds for $g=1$ and $g=2$, that is for $\{0,2,\to\}=\langle 2,3\rangle$ and $\{0,2,4,\to\}=\langle 2,5\rangle$.
\end{remark}

\begin{example}
    Let $c$ be a positive integer and let $S=\{0\}\cup(c+\mathbb{N})$. In this case, the type of $S$ equals the genus of $S$, and every set $X$ of pseudo-Frobenius numbers gives rise to an ideal $I=(\{ 0 \} \cup X)+S$. Then the size of the class monoid is precisely two to the power of the type of $S$ and the minimal bound is attained for $S$.
\end{example}

We can improve a little bit the lower bound given in Proposition~\ref{prop:lo-up-bound-clS-genus-type}. 

\begin{proposition} \label{improved}
    Let $S$ be a numerical semigroup with multiplicity $m$ and genus $g$. Then
    \[ 2^{m-1}+g-m+1\le |\mathcal{C}\ell(S)|.
    \]
\end{proposition}
\begin{proof}
    The proof follows by considering antichains in the Hasse diagrams of $\operatorname{G}(S)$ of the form $\{x\}$ with $x$ a gap larger than $m$, and the antichains formed by gaps in $\{1,\dots,m-1\}$.
\end{proof}

This lower bound is better since, by \eqref{eq:pF-max}, the type of a numerical semigroup is at most its multiplicity minus one.

\begin{example}
    There are three numerical semigroups of genus 10 attaining this bound: $\langle 11,\dots, 21 \rangle$, $\langle  10, 11, 12, 13, 14, 15, 16, 17, 18 \rangle$, and $\langle 2, 21 \rangle$.
\end{example}

\section{Apéry sets}
Denote by $m$ the multiplicity of $S$ and define 
\[
\operatorname{Ap}(E,m)=\{w_0(E),\dots,w_{m-1}(E)\},
\]
where $w_i(E)$ is the minimum element in $E$ congruent with $m$ modulo $i$. It follows easily that \[\operatorname{Ap}(E,m)=\{ e\in E : e-m\not\in S\}\]
and that $E=\operatorname{Ap}(E,m)+S$. Notice that the definition of Apéry set for an ideal is slightly different to that given in \cite{v1f}.

\begin{remark}
Let $I,J\in \mathcal{I}_0(S)$. Then $I\cap J$ and $I\cup J$ are also in $\mathcal{I}_0(S)$. Moreover,
\[
\Ap(I\cap J,m)= \{ 0, \max\{w_1(I),w_1(J)\},\dots, \max\{w_{m-1}(I),w_{m-1}(J)\}\},
\]
and 
\[
\Ap(I\cup J,m)= \{ 0, \min\{w_1(I),w_1(J)\},\dots, \min\{w_{m-1}(I),w_{m-1}(J)\}\}.
\]
Also $I\subseteq J$ if and only if, component-wise,  $(w_0(J),\dots,w_{m-1}(J))\le (w_0(I),\dots, w_{m-1}(I))$.
\end{remark}

\begin{example}
Let $S=\langle 5,7,9\rangle$, $I=\{0,2\}+S$ and $J=\{0,3,4\}+S$. Then 
\begin{itemize}
    \item $\operatorname{Ap}(I,5)=\{ 0, 11, 2, 18, 9 \}$,
    \item $\operatorname{Ap}(J,5)=\{ 0, 11, 7, 3, 4\}$,
    \item $\operatorname{Ap}(I\cap J,5)=\{ 0, 11, 7, 18, 9\}$,
    \item $\operatorname{Ap}(I\cup J,5)=\{ 0, 11, 2, 3, 4\}$.
\end{itemize}
\end{example}

Notice that $S$ is itself an ideal of $S$, and that $\Ap(S,m)$ coincides with the Apéry set of $m$ in $S$ in the usual sense. It is well known that every element $s$ in $S$ can be expressed uniquely as $s=km+w$ with $k\in \mathbb{N}$ and $w\in \Ap(S,m)$. The following result characterizes those sets that are Apéry sets of ideals of a numerical semigroup (see \cite[Lemma~8]{kc} for a similar result for numerical semigroups). 

\begin{lemma}\label{lem:ap}
Let $S$ be a numerical semigroup with multiplicity $m$ and let $A=\{w_0=0,w_1,\dots,w_{m-1}\}\subset \mathbb{N}$ be such that $w_i\equiv i \pmod m$ for all $i\in \{0,\dots,m-1\}$. Then $A=\Ap(E,m)$ for some ideal $E$ of $S$ with $\min(E)=0$ if and only if for all $i,j\in\{0,\dots,m-1\}$, we have $w_i+w_j(S)\ge w_{(i+j)\!\mod m}$.
\end{lemma}
\begin{proof}
For sake of simplicity, denote $\overline{i+j}=(i+j)\mod m$.

Suppose that $E$ is an ideal of $S$ with $\min(E)=0$. Then $w_i(E)+w_j(S)\in E$, and consequently $w_i(E)+w_j(S)\ge w_{\overline{i+j}}(E)$. 
Now suppose that $A$ is a set fulfilling the inequalities of the statement. Let $E=A+S$. Then $w_i(E)\le w_i$ for all $i\in \{0,\dots,m-1\}$. As $w_i(E)\in E$, we have that $w_i(E)=w_j+s$ for some $j\in \{0,\dots,m-1\}$ and $s\in S$. There exists $t\in \mathbb{N}$ and $k\in \{0,\dots,m-1\}$ such that $s=tm+w_k(S)$, whence $w_i(E)=w_j+t m+w_k(S)$. By the standing hypothesis $w_j+w_k(S)\ge w_{\overline{j+k}}$. Notice that $\overline{j+k}=i$, and so $w_i(E)\ge t m + w_{i}\ge w_i$, forcing $w_i=w_i(E)$. 
\end{proof}

Let $S$ be a numerical semigroup with multiplicity $m$, and let $E$ be an ideal of $S$ with $\min(E)=0$. Recall that the Kunz coordinates of $S$ are the $(m-1)$-uple $(k_1(S),\dots,k_{m-1}(S))$ such that that $w_i(S)=k_i(S)m+i$ for all $i\in \{1,\dots,m-1\}$. We can proceed similarly with $E$.
For every $i\in \{0,\dots,m-1\}$, there exists $k_i(E)\in \mathbb{N}$ such that $w_i(E)=k_i(E)m+i$. The $(m-1)$-uple $(k_1(E),\dots,k_{m-1}(E))$ is known as the Kunz coordinates of $E$.

The inequalities in Lemma~\ref{lem:ap} imply that for all $i,j\in \{0,\ldots,m-1\}$, $w_i(E)+w_j(S)\ge w_{\overline{i+j}}(E)$, and so 
\begin{itemize}
    \item $k_i(E)\le k_i(S)$,
    \item $k_i(E)+k_j(S)\ge k_{i+j}(E)$ if $i+j<m$,
    \item $k_i(E)+k_j(S)\ge k_{i+j-m}(E)-1$ if $i+j\ge m$.
\end{itemize}
The first inequality follows from $w_0(E)+w_j(S)\ge w_{j}(E)$. Notice that we have the correspondence $w_i(E)=k_i(E)m+i$. So if $E$ is an ideal, its Kunz coordinates will fulfill the above inequalities, and if we have a tuple $(k_1,\dots,k_{m-1})$ fulfilling these inequalities, then the set $\{0,k_1m+1,\dots, k_{m-1}m+m-1\}$ is the Apéry set of an ideal by Lemma~\ref{lem:ap} (actually the ideal $\{0,k_1m+1,\dots, k_{m-1}m+m-1\}+S$).
In light of the above discussion we can state the following result.

\begin{theorem}\label{kunz}
Let $S$ be a numerical semigroup with multiplicity $m$ and Kunz coordinates $(k_1,\dots,k_{m-1})$. The set of ideals $E$ of $S$ with $\min(E)=0$ are in one-to-one correspondence with the set $\mathcal{K}(S)$ of solutions of the following system of inequalities over the set of non-negative integers:
\[
\begin{array}{l}
x_i\le k_i, \text{ for all } i\in \{1,\dots,m-1\},\\
x_{i+j}-x_i\le k_j, \text{ for every } i,j\in\{1,\dots,m-1\}, i+j<m \\
x_{i+j-m}-x_i\le k_j+1, \text{ for every } i,j\in\{1,\dots,m-1\}, i+j > m.
\end{array}
\]
\end{theorem}

In particular, from the first set of inequalities we get the following bound.

\begin{corollary}\label{cor:bound-ideals-kunz}
Let $S$ be a numerical semigroup with multiplicity $m$ and Kunz coordinates $(k_1,\dots,k_{m-1})$. Then \[|\mathcal{I}_0(S)|\le \prod_{i=1}^{m-1} (k_i+1).\]
Equality holds  if and only if $S=\langle m\rangle \cup (c+\mathbb{N})$, with $c$ a positive integer greater than $m$, that is,
\begin{enumerate}
    \item $k_1\ge \dots \ge k_{m-1}$, and
    \item $k_1-k_{m-1}\le 1$.
\end{enumerate}
\end{corollary}
\begin{proof}
The inequality follows directly from $0\le k_i(E)\le k_i(S)$. So, let us focus on when the equality holds.

Observe that if $S=\{0,m,2m,\dots, km, c,\to\}$ and we write $c=km+i$ for $i\in\{0,\dots,m-1\}$, we obtain $k_i=\dots=k_{m-1}=k$, and $k_1=\dots=k_{i-1}=k+1$. The converse is also easy to prove.

\emph{Necessity}. Notice that $|\mathcal{I}_0(S)|=\prod_{i=1}^{m-1}(k_i+1)$ forces every tuple $(k_1',\dots,k_{m-1}')$ with $0\le k_i'\le k_i$ for all $i$ to be the Kunz coordinates of an ideal of $S$. In particular, the tuple $(0,\dots,0,k_{i+1},0,\dots,0)$ ($k_{i+1}$ in the $(i+1)$th coordinate and zero in the rest) corresponds with the Kunz coordinates of some ideal $E$ of $S$. Then $k_1(E)+k_i(S)=0+k_i\ge k_{i+1}(E)=k_{i+1}$, that is, $k_i\ge k_{i+1}$. This proves that $k_1\ge \dots \ge k_{m-1}$.

As $(k_1,0,\dots,0)$ are the Kunz coordinates of some ideal of $S$, say $E$, we deduce that $k_2(E)+k_{m-1}(S)\ge k_1(E)-1$, that is, $k_{m-1}\ge k_1(E)-1$, and consequently $k_1-k_{m-1}\le 1$, proving in this way the second assertion.

\emph{Sufficiency}. Let $(k_1',\dots,k_{m-1}')$ be a tuple with $k_i'\le k_i$ for all $i$. Notice that $k_i'+k_j\ge k_i'\ge k_{i+j}'$ if $i+j<m$. If $i+j>m$, then $k_i'+k_j\ge k_{i+j-m}'-1$ if and only if $k_j\ge k_{i+j-m}'-k_i'-1$, which clearly holds since $|k_{i+j-m}'-k_i'|\le 1$.  
\end{proof}

\begin{example}
Let $S=\langle 5,6,8,9\rangle$. The Kunz coordinates of $S$ are $(1,2,1,1)$, and so Corollary~\ref{cor:bound-ideals-kunz} ensures that the number of elements in $\mathcal{I}_0(S)$ is at most 24. The cardinality of $\mathcal{I}_0(S)$ is 20. 
For $S=\langle 3,5,7\rangle$, with Kunz coordinates $(2,1)$, the bound is sharp.
\end{example}

\begin{remark}
Notice that the genus of $S$ is precisely $k_1+\dots+k_{m-1}$. Thus $2^g=\prod_{i=1}^{m-1}2^{k_i}\ge\prod_{i=1}^{m-1} (k_i+1)$.
\end{remark}

Let $S$ be a semigroup with multiplicity $m$, and let $E\in \mathcal{I}_0(S)$. From now on, and in order to ease the notation, for every integer $i$, we will write 
\[ 
w_i(E)=w_{i\,\mathrm{mod}\, m}(E),
\]
and we will do the same with the elements in the Apéry set of $S$ with respect to $m$.

\begin{proposition}\label{prop:ap-sum}
Let $S$ be a numerical semigroup with multiplicity $m$, and let $I$, $J$ be two ideals in $\mathcal{I}_0(S)$. Then, for every $i\in\{0,\dots,m-1\}$,
\[
w_i(I+J)=\min\{ w_{i_i}(I)+w_{i_2}(J) : i_1,i_2\in \{0,\dots,m-1\}, i_i+i_2\equiv i \pmod m\}.
\]
\end{proposition}
\begin{proof}
Notice that if $i_1+i_2\equiv i \pmod m$, then $w_{i_1}(I)+w_{i_2}(J)\equiv i \pmod m$. Hence $w_{i_1}(I)+w_{i_2}(J)\ge w_i(I+j)$. As $w_i(I+J)\in I+J$, and $I=\Ap(I,m)+S$ and $J=\Ap(J,m)+S$, there exists $s_1,s_2\in S$ and $i_1,i_2\in \{0,\dots,m-1\}$ such that $w_i(I+J)=w_{i_1}(I)+s_1+w_{i_2}(J)+s_2$. Let $j_1,j_2\in \{0,\ldots,m-1\}$ and $t_1,t_2\in \mathbb{N}$ be such that $s_1=t_1m+w_{j_1}(S)$ and $s_2=t_2m+w_{j_2}(S)$. Then $w_i(I+j)=w_{i_1}(I)+w_{j_1}(S)+w_{i_2}(J)+w_{j_2}(S)+(t_1+t_2)m$. By using Lemma~\ref{lem:ap}, we deduce that $w_i(I+J)\ge w_{i_1+j_1}(I)+w_{i_2+j_2}(J)+(t_1+t_2)m\ge w_{i_1+j_1}(I)+w_{i_2+j_2}(J)$. Notice that $i\equiv (i_1+j_1+i_2+j_2) \mod m$, and so  $w_i(I+J)\ge \min\{ w_{i_i}(I)+w_{i_2}(J) : i_1,i_2\in \{0,\dots,m-1\}, i_i+i_2\equiv i \pmod m\}$.
\end{proof}

With this proposition, one can endow $\mathcal{K}(S)$ with the operation $x+x'=:y$ with $y_i=\min\{ x_j+x_k'-(j+k-i)/m : j,k\in \{0,\dots,m-1\}, j+k\equiv i \mod m \}$. This monoid is hence isomorphic to $\mathcal{I}_0(S)$.

\begin{example}
Let $S=\langle 3,5,7\rangle$. The following table is the addition table of $\mathcal{I}_0(S)$. We represent each ideal by its Apéry set with respect to the multiplicity of $S$. The identity element (and the minimum) is $0+S=\{0,7,5\}+S$, while $\{0,1,2\}+S=\mathbb{N}$ is the maximal ideal and acts as sink or infinity, that is, every time we add an ideal to $\mathbb{N}$ we obtain again $\mathbb{N}$. 
\[
\begin{array}{l|llllll}
      +      & \{ 0, 7, 5 \} & \{ 0, 4, 5 \} & \{ 0, 7, 2 \} & \{ 0, 1, 5 \} & \{ 0, 4, 2 \} & \{ 0, 1, 2 \}\\ \hline
\{ 0, 7, 5 \} & \{ 0, 7, 5 \} & \{ 0, 4, 5 \} & \{ 0, 7, 2 \} & \{ 0, 1, 5 \} & \{ 0, 4, 2 \} & \{ 0, 1, 2 \}\\
\{ 0, 4, 5 \} & \{ 0, 4, 5 \} & \{ 0, 4, 5 \} & \{ 0, 4, 2 \} & \{ 0, 1, 5 \} & \{ 0, 4, 2 \} & \{ 0, 1, 2 \}\\
\{ 0, 7, 2 \} & \{ 0, 7, 2 \} & \{ 0, 4, 2 \} & \{ 0, 4, 2 \} & \{ 0, 1, 2 \} & \{ 0, 4, 2 \} & \{ 0, 1, 2 \}\\
\{ 0, 1, 5 \} & \{ 0, 1, 5 \} & \{ 0, 1, 5 \} & \{ 0, 1, 2 \} & \{ 0, 1, 2 \} & \{ 0, 1, 2 \} & \{ 0, 1, 2 \}\\
\{ 0, 4, 2 \} & \{ 0, 4, 2 \} & \{ 0, 4, 2 \} & \{ 0, 4, 2 \} & \{ 0, 1, 2 \} & \{ 0, 4, 2 \} & \{ 0, 1, 2 \}\\
\{ 0, 1, 2 \} & \{ 0, 1, 2 \} & \{ 0, 1, 2 \} & \{ 0, 1, 2 \} & \{ 0, 1, 2 \} & \{ 0, 1, 2 \} & \{ 0, 1, 2 \}
\end{array}
\]
\end{example}

\subsection{The canonical ideal}

Let $S$ be a numerical semigroup. The standard canonical ideal of $S$ is 
\[\operatorname{K}(S)=\{ x\in \mathbb{Z} : \operatorname{F}(S)-x\not\in S\}.\]
The following result recovers the description of the Apéry set of the canonical ideal of a numerical semigroup given in \cite[Proposition~1.3.9]{tesi-serban}.

\begin{proposition}\label{prop:canonical-apery}
    Let $S$ be a numerical semigroup with multiplicity $m$, and let $E$ be in $\mathcal{I}_0(S)$. Let $f=\operatorname{F}(S)\mod m$. Then $E=\operatorname{K}(S)$ if and only if $w_i(E)+w_j(S)=w_f(S)=w_f(E)$ for all $i,j\in \{0,\dots,m-1\}$ with $i+j\equiv f\pmod m$.
\end{proposition}
\begin{proof}
Set $K=\operatorname{K}(S)$. Notice that $\operatorname{F}(S)=\operatorname{F}(K)$ and thus $w_f(S)=w_f(K)$, and $\operatorname{F}(S)=w_f(S)-m$ (see for instance \cite[Proposition~2.12]{ns}). 

As $w_i(K)-m\not\in K$, we have that $\operatorname{F}(S)-(w_i(K)-m)=w_f(S)-w_i(K)\in S$, which means that $w_f(S)-w_i(K)\ge w_j(S)$, $i+j\equiv f\mod m$, or equivalently, $w_i(K)+w_j(S)\le w_f(S)$, and by Lemma~\ref{lem:ap}, we know that $w_i(K)+w_j(S)\ge w_f(S)$. Hence $w_i(K)+w_j(S)= w_f(S)$.

Now suppose that $E$ is an ideal in $\mathcal{I}_0(S)$ with $w_f(E)=w_f(S)$ and $w_i(E)+w_j(S)=w_f(E)$ for all $i,j\in \{0,\dots,m-1\}$ with $i+j\equiv f\pmod m$. Let $x\in K$. Let $i= x \mod m$, and $j=(f-i) \mod m$. Then $\operatorname{F}(S)-x\not\in S$, which means that $w_f(S)-m-x<w_j(S)$, and so $w_j(S)+x+m>w_f(S)=w_i(E)+w_j(S)$. It follows that $x+m>w_i(E)$, and so $x\ge w_i(E)$, yielding $x\in E$. For the other inclusion, let us prove that $w_i(E)$ is in $K$ for all $i$. Let $j$ be as above. We have to show that $\operatorname{F}(S)-w_i(E)\not\in S$, or equivalently, $w_f(S)-m-w_i(E)<w_j(S)$, which translates to $w_i(E)+w_j(S)+m>w_f(E)$. But this trivially holds, since by hypothesis $w_i(E)+w_j(S)=w_f(E)$.
\end{proof}

With this, we retrieve the following result, that is most probably known, but for which we do not find an appropriate reference in the literature.

\begin{corollary} \label{known}
    Let $S$ be a numerical semigroup. Then the canonical ideal of $S$ is generated by $\{\operatorname{F}(S)-g : g \in \operatorname{PF}(S)\}$.
\end{corollary}
\begin{proof}
Let $K$ be the canonical ideal of $S$. Let $m$ be the multiplicity of $S$ and $f$ the Frobenius number of $S$. Let us prove that 
$K=\{f-g : g \in \operatorname{PF}(S)\}+S$.

By Proposition~\ref{prop:canonical-apery}, we know that $w_i(K)=w_f(S)-w_{f-i}(S)$ for all $i$; whence $w_i(K)=(f+m)-w_{f-i}(S)$. We also know that $K=\operatorname{Ap}(K,m)+S$. For every $i$, there exists $m_i\in \operatorname{Maximals}_{\le_S}(\operatorname{Ap}(S,m))$ and $s_i\in S$ such that $w_{f-i}(S)+s_i=m_i$. Hence $w_i(K)=(f+m)-(-s_i+m_i)=f+m+s_i-m_i=(f-(m_i-m))+s_i$. By \eqref{eq:pF-max}, $m_i-m\in \operatorname{PF}(S)$, and so $w_i(K)\in \{f-g : g \in \operatorname{PF}(S)\}+S$. This proves that $K\subseteq \{f-g : g \in \operatorname{PF}(S)\}+S$. Now let $g\in \operatorname{PF}(S)$. Then $f-(f-g)=g\not\in S$, and thus $f-g\in K$, and consequentely $\{f-g : g \in \operatorname{PF}(S)\}+S\subseteq K$. 
\end{proof}

\subsection{Reduction number}

If in Proposition~\ref{prop:ap-sum}, we take $J=I$, we obtain 
\[
w_i(I+I)=\min\{ w_{i_i}(I)+w_{i_2}(I) : i_1,i_2\in \{0,\dots,m-1\}, i_i+i_2\equiv i \pmod m\}.
\]
Since $I\subseteq  I+I = 2I$ and $\mathcal{I}_0(S)$ has finitely many elements, we deduce that there is some $r$ such that $(r+1)I=r I$. This $r$ is known as the reduction number of $I$, denoted by $\operatorname{r}(I)$. 

Notice that for all $i$, $w_i((r+1)I)\le w_i(r I)$, and $r$ is the reduction number of $I$ precisely when all these inequalities become equalities. 

Some basic properties of the reduction number of an ideal can be found in \cite{b-k}, and we summarize them in the next result.

\begin{proposition}\label{prop:basic-reduction}
    Let $S$ be a numerical semigroup with multiplicity $m$.
    \begin{enumerate}
        \item For every ideal $E$ of $S$ and every positive integer $j\le \operatorname{r}(E)$, we have that $\nu(jE)>j$.
        \item For all $E\in \mathcal{I}_0(S)$, $\operatorname{r}(E)\le m-1$.
        \item For every $r\in \{1,\dots,m-1\}$, there exists $E\in \mathcal{I}_0(S)$ with $\operatorname{r}(E)=r$.
    \end{enumerate}        
\end{proposition}

\begin{example}
Let $S$ be a numerical semigroup with multiplicity $m$, and let $g\in \mathbb{N}\setminus S$. Consider the ideal $E=\{0,g\}+S$. Observe that $w_i(E)=w_i(S)$ for all $i\not\equiv g \pmod m$, and that $w_{g}(E)=g<w_{g}(S)$. Clearly, $E+E=\{0,g,2g\}+S$, and so $E+E=E$ if and only if $2g\in S$. Thus $\operatorname{r}(\{0,g\}+S)=1$ if and only if $2g\in S$.
\end{example}

In general, we obtain the following (compare with \cite[Proposition~2.3.9]{b-k}).
\begin{proposition}\label{prop:red-two-gens}
Let $S$ be a numerical semigroup, and let $g$ be a gap of $S$. Then
\[\operatorname{r}(\{0,g\}+S)=\min\{ k\in \mathbb{N} : (k+1)g\in S\}.\] 
\end{proposition}
\begin{proof}
Let $E=\{0,g\}+S$. It is easy to see that $k E=\{0,g,\dots,kg\}+S$ for $k$ a positive integer. 

Let $r=\operatorname{r}(E)$ and let $t=\min\{k\in \mathbb{N} : (k+1)g\in S\}$. As $(t+1)g\in S$, we have that $(t+1)E=\{0,\dots,tg\}+S=t E$, and so $r\le t$. We also know that $(r+1)E=r E$ and that $(k+1)E\neq k E$ for all $k<r$. The equality $(r+1)E=r E$ forces $(r+1)g$ to be in $r E$, and so there exists $j\in\{0,\dots,r\}$ such that $(r+1)g\in j g+S$. If $j=0$, then $(r+1)g\in S$ and the minimality of $t$ yields $t\le r$. If $j>0$, we get $(r+1-j)g\in S$, and this forces $(r+1-j) E=(r-j) E$, which is impossible. 
\end{proof}

\begin{example}
    Let $n$ be an odd integer greater than one. Let $S=\langle 2,n\rangle$. Every ideal $E$ in $\mathcal{I}_0(S)\setminus\{S\}$ is of the form $\{0,k\}+S$, with $k$ an odd integer smaller than $n$ (a gap of $S$). Notice that $2k\in S$, and so according to Proposition~\ref{prop:red-two-gens}, $\operatorname{r}(E)=1$. Note that this also follows from Proposition~\ref{prop:basic-reduction}~(2), since in this setting $m=2$. 
    
    Notice that, by Proposition~\ref{prop:basic-reduction}~(3), if $m>2$, then there is an ideal $E$ such that $\operatorname{r}(E)>1$. Thus, if every nontrivial ideal has reduction number one, the multiplicity must be at most two.

    Observe also that $2E=E$ for all $E\in \mathcal{I}_0(S)\setminus\{S\}$ implies that $E=2E+S$, and so $E$ is in $2E+\mathcal{I}_0(S)$ (that is, $\mathcal{I}_0(S)$ is a Clifford semigroup). The converse is also true. If $\mathcal{I}_0(S)$ is a Clifford semigroup, that is, for all $E\in \mathcal{I}_0(S)$, $E\in 2E+\mathcal{I}_0(S)$, then there exists $E'\in \mathcal{I}_0(S)$ such that $E=2E+E'$, but then $E\subseteq 2E\subseteq 2E+E'=E$, which means that $E=2E$.
\end{example}

Next we give an upper bound for the reduction number of ideals generated by sets of the form $\{0\}\cup X$ with $X$ a set of gaps smaller than the multiplicity of the semigroup.

\begin{proposition}\label{pr:red-gaps-smaller-mult}
Let $S$ be a numerical semigroup with multiplicity $m$. If $E= \{ 0, a_1, \dots, a_{h} \}+S\in \mathcal{I}_0(S)$, with $a_i \leq m$ for every $i$, then $\operatorname{r}(E) \leq m-h$. 
\end{proposition}
\begin{proof}
We can assume $h<m-1$, otherwise $E=\mathbb N$ and there is nothing to prove. Moreover, for $h=1$ the thesis follows immediately by Proposition~\ref{prop:basic-reduction}, so we can assume $h>1$.

Set $a_0=0$ and $x=m-h+1$. We want to show that $x E=(x-1)E$. So take an element $t$ in $x E$. We have $t=a_{i_1}+\dots +a_{i_x}+s$, with $a_{i_j}\in \{a_1, \dots, a_{h}\}$ and $s \in S$. Set $u=a_{i_1}+\dots +a_{i_x}$. If $u \equiv a_i \pmod{m}$ for some $i\in \{0,\dots, h\}$, then, since $a_i < m$, we deduce $u = a_i + y m$, with $y\in\mathbb{N}$, which means means $t=u+s=a_i+y m+s \in E \subseteq (x-1)E$.

For every $k\in \{2,\dots,x\}$, set $u_k=a_{i_1}+\dots +a_{i_k}$. As above, if $u_k \equiv a_i \pmod{m}$ for some $i$, then $u_k=a_i+y m$, with $y\in \mathbb{N}$, since $a_i < m$ for every $i$. Thus, $t=u+s=a_i+y m+a_{i_{k+1}}+\dots +a_{i_x} +s \in (x-k+1) E$. As $2\leq x-k+1 \leq x-1$, we get $t\in (x-1)E$.

Hence we may suppose that $u$ and every $u_k$ are not congruent to any $a_i$ modulo $m$. By setting $u_x=u$, since $|\{u_2, \dots, u_{x-1},u_x\}|=x-1=m-h$ and $|\{a_0, \dots, a_h\}|=h+1$, there exist two partial sums $u_z, u_w$, with $2 \leq w < z \leq x$, such that $u_z \equiv u_w \pmod{m}$. It follows that $t=u+s=a_{i_1}+\dots +a_{i_w}+y m+a_{i_{z+1}}+\dots +a_{i_x}+s \in (w+x-z)E$ (with the obvious convention that if $z=x$, then $a_{i_{z+1}}+\dots +a_{i_x}=0$). As $2 \leq w+x-z \leq x-1$, we obtain once more that $t\in (x-1)E$ as desired.
\end{proof}

\begin{figure}
    \centering
    \includegraphics[scale=0.5]{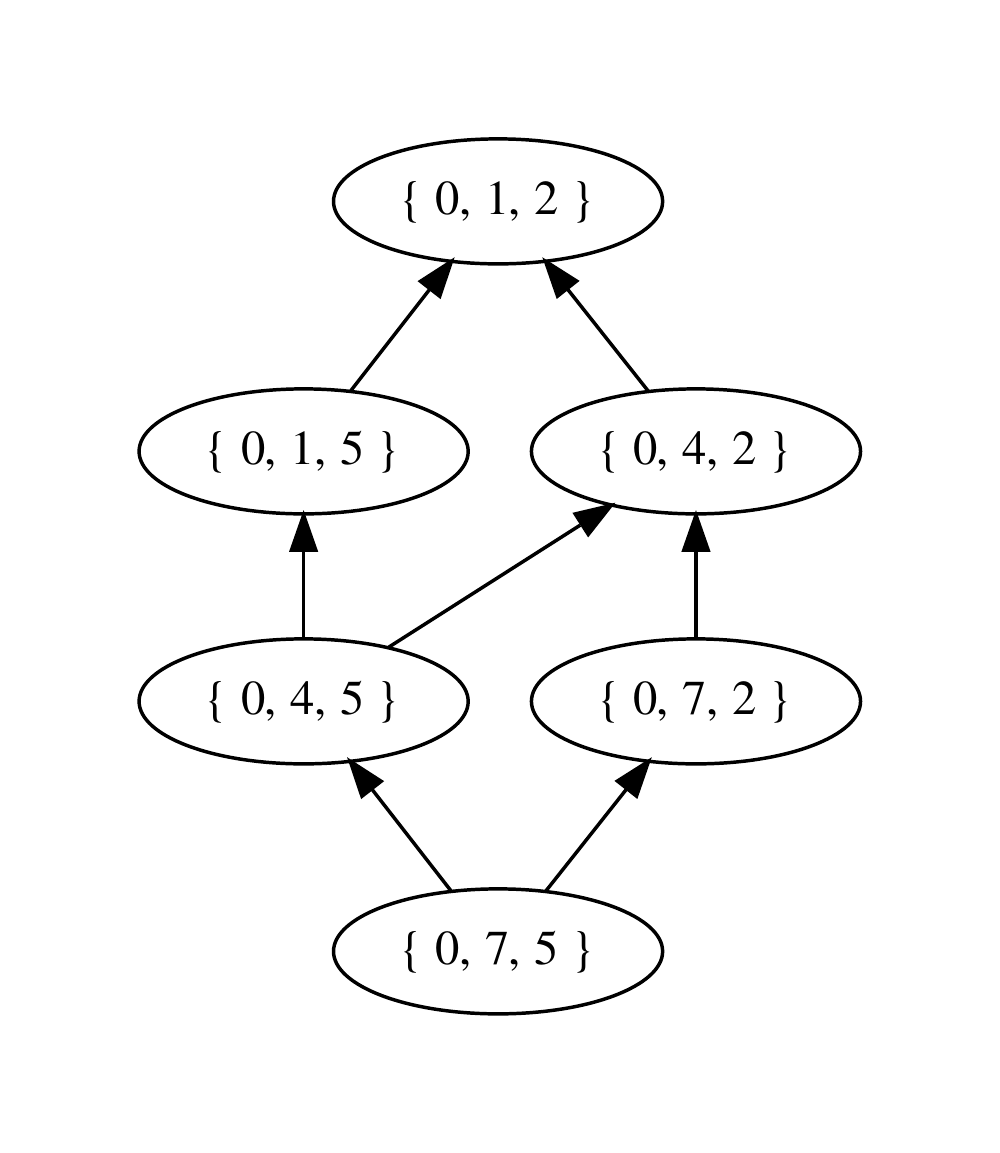}
\caption{Hasse diagram of $\mathcal{I}_0(\langle 3,5,7\rangle)$ arranged by inclusion}
    \label{fig:id-3-5-7-u}
\end{figure}

\subsection{Hasse diagram of the class monoid}
By using Apéry sets, we can obtain some information about the Hasse diagram (with respect to inclusion, see Figure~\ref{fig:id-3-5-7-u}) of the ideal class monoid of a numerical semigroup.
 
\begin{proposition}\label{prop:mins-maxs}
Let $S$ be a numerical semigroup with multiplicity $m$. Then
\begin{enumerate}
    \item $\max_{\subseteq}(\mathcal{I}_0(S))=\{0,\dots,m-1\}+S=\mathbb{N}$,
    \item $\min_{\subseteq}(\mathcal{I}_0(S))=S$,
    \item $\operatorname{Maximals}_{\subseteq}(\mathcal{I}_0(S)\setminus\{\mathbb{N}\})=\big\{ \{0,1,\dots,i-1,i+m,i+1,\dots,m-1\}+S : i\in \{1,\dots,m-1\}\big\}$,
    \item $\operatorname{Minimals}_{\subseteq}(\mathcal{I}_0(S)\setminus\{S\})=\left\{ \{0,f\}+S : f\in \operatorname{PF}(S)\right\}$,
    \item the largest positive integer $t$ for which there exists a strictly ascending chain $I_0\subsetneq I_1\subsetneq \dots \subsetneq I_t$ of ideals of $\mathcal{I}_0(S)$ is $\operatorname{g}(S)$.
\end{enumerate}
\end{proposition}
\begin{proof}
Recall that for $I,J\in \mathcal{I}_0(S)$, $I\subseteq I$ if and only if $(0,w_1(J),\dots,w_{m-1}(J))\le (0,w_1(I),\dots,w_{m-1}(I))$. Notice that whenever $w_i(J)\le w_i(I)$, we also have that $w_i(I)-w_i(J)$ is a multiple of $m$. 

The first two assertions follow directly from the definitions. As for the third, notice that for $\mathbb{N}=\{0,1,\dots,m-1\}+S$, we have that $w_i(\mathbb{N})=i$ for all $i$. Thus the potential candidates to be maximal below $\mathbb{N}$ are precisely those whose Apéry sets with respect to $m$ differ in one element. Let us prove that for every $i$, the set $\{0,1,\dots,i-1,i+m,i+1,\dots,m-1\}$ is an Apéry list of an ideal in $\mathcal{I}_0(S)$. By using Lemma~\ref{lem:ap}, this reduces to showing that for every $k\in \{0,\ldots,m-1\}$, $j+w_k(S)\ge (j+k)\mod m$ for $j\neq i$ and that $i+m+w_k(S)\ge (i+k)\mod m$. The first inequality holds because $j+w_k(S)\ge j+k$ and $(j+k)\mod m$ is either $j+k$ or $j+k-m$. The second inequality also holds, because $i+m+w_k(S)>i+k$. 

Now let us focus on the fourth assertion of the statement. For the same reason as in the previous paragraph, the candidates to be minimal above $S$ are those with Apéry lists of the form $\{0,w_1(S),\dots,w_{i-1}(S),w_i(S)-m,w_{i+1}(S),\dots,w_{m-1}(S)\}$. With the use of Lemma~\ref{lem:ap}, let us check which of these sets are Apéry sets of ideals in $\mathcal{I}_0(S)$. For all $j\neq i$ and for all $k$, we have that $w_j(S)+w_k(S)\ge w_{j+k}(S)$, if $(j+k)\mod m$ is not equal to $i$, since $S$ is a numerical semigroup. If $(j+k)\mod m$ is $i$, then we also have $w_j(S)+w_k(S)\ge w_i(S)\ge w_i(S)-m$. Also $w_i(S)-m+0\ge w_i(S)-m$, so it remains to check whether $w_i(S)-m+w_k(S)\ge w_{i+k}(S)$ for all $k\neq 0$. Notice that $(w_i(S)-m+w_k(S))\mod m=i+k$, and consequently $w_i(S)-m+w_k(S)\ge w_{i+k}(S)$ if and only if $w_i(S)-m+w_k(S)\in S$. Observe that $w_i(S)-m+w_k(S)\in S$ for all $k\neq 0$ if and only if $w_i(S)-m+s\in S$ for all $s\in S\setminus\{0\}$. Thus, $\{0,w_1(S),\dots,w_{i-1}(S),w_i(S)-m,w_{i+1}(S),\dots,w_{m-1}(S)\}$ is an Apéry list of an ideal in $\mathcal{I}_0(S)$ if and only if $w_i(S)-m\in \operatorname{PF}(S)$. Finally, observe that \[\{0,w_1(S),\dots,w_{i-1}(S),w_i(S)-m,w_{i+1}(S),\dots,w_{m-1}(S)\}+S=\{0,w_i(S)-m\}+S.\]

\item Let $I_0\subsetneq I_1\subsetneq \dots \subsetneq I_t$ be an ascending chain of ideals of $\mathcal{I}_0(S)$. Let $G_i=I_{i+1}\setminus I_i$. Then $G_i$ is a set of gaps, and $G_0\subsetneq G_0\cup G_1 \subsetneq \dots \subsetneq G_0\cup \dots \cup G_{t-1}$ is an increasing sequence of sets of gaps. In particular, $t\le \operatorname{g}(S)$. Now write $\mathbb{N}\setminus S=\{g_1>\dots>g_{\operatorname{g}(S)}\}$. Then the sequence of ideals $I_i=\{0,g_1,\dots,g_i\}+S$ is an increasing sequence of ideals (with respect to inclusion).
\end{proof}

\begin{remark}\label{rem:type-mins}
Notice that as consequence of Proposition~\ref{prop:mins-maxs}~(4), we recover the fact that every ideal in $\mathcal{I}_0(S)\setminus\{S\}$ contains a pseudo-Frobenius number (see the proof of Proposition~\ref{prop:lo-up-bound-clS-genus-type}). It also follows that \[\operatorname{t}(S)=|\operatorname{Minimals}_\subseteq (\mathcal{I}_0(S)\setminus\{S\})|.\]
\end{remark}

\begin{remark}\label{width}
We now want to find a lower bound for the width of the Hasse diagram of $\mathcal{I}_0(S)$ with respect to inclusion. This can be achieved by looking at the ``second level'' of the Hasse diagram (minimal non-zero ideals) yielding $\operatorname{m}(S)-1$, or by using the pigeonhole principle: by removing the maximum ($\mathbb{N}$) and minimum ($S$) from the Hasse diagram, having $\operatorname{g}(S)-1$ remaining ``levels'' we have that the width will be at least $\lceil(|\mathcal{I}_0(S)|-2)/(\operatorname{g}(S)-1)\rceil$. This bound is sharp and is attained for example in the case $S=\langle 3,5\rangle$, for which the width is 2. 
\end{remark}

\begin{example} Using the description of maximal elements in $\mathcal{I}_0(S)\setminus\{\mathbb{N}\}$, we can compute their reduction number. More precisely,
let $S$ be a numerical semigroup with multiplicity $m$, and let $E$ be a maximal element of $\mathcal{I}_0(S)\setminus\{\mathbb{N}\}$, that is, $E=\{0,1,\dots, i-1, i+m,i+1,\dots,m-1\}+S$. Notice that $E\subseteq E+E\subseteq \mathbb{N}$, and that $E\neq E+E$ if and only if $i\not\in E+E$. Consequently,
\[
\operatorname{r}(\{0,1,\dots, i-1, i+m,i+1,\dots,m-1\}+S)=\begin{cases}
    1 & \text{if } i=1,\\
    2 & \text{otherwise.}
\end{cases}
\]
\end{example}

\section{Irreducible elements, atoms, quarks and primes}\label{quarks}
Let $H$ be a monoid (in our setting commutative, and thus we use additive notation). For $a, b\in H$ define $a\preceq b$ if there exists $c\in H$ such that $b=a+c$. We write $a\prec b$ whenever $a\preceq b$ and $b\not\preceq a$. We say that $a\in H$ is a unit if there exists $b\in H$ such that $a+b=0$ (the identity element of $H$). Following \cite{t} (for our specific $\preceq$, $\preceq$-units are units), a non unit $a$ is said to be 
\begin{itemize}
    \item irreducible if $a\neq x + y$ for all non-units $x$ and $y$ of $H$ such that $x\prec a$ and $y\prec a$;
    \item an atom if $a\neq x + y$ for all non-units $x$ and $y$ of $H$;
    \item a quark if there is no non-unit $b$ with $b\prec a$;
    \item a prime if $a\preceq x + y$ for some $x,y\in H$ implies that $a\preceq x$ or $a\preceq y$.
\end{itemize}

Notice that $\mathcal{I}_0(S)$ is reduced (Proposition~\ref{prop:reduced}), that is, the only unit is its identity element, which in this case is $0+S=S$. Observe that $\mathcal{I}_0(S)$ is commutative, but it is not cancellative. Also if $I\preceq J$ in $\mathcal{I}_0(S)$, then $J=I+K$ for some ideal $K\in\mathcal{I}_0(S)$, whence $I\subseteq J$. Thus $I\preceq J\not\preceq I$ implies that $I\subsetneq J$, since $I=J$ would imply $J\preceq I$. This shows that $I\prec J$ implies that $I\preceq J$ and $I\neq J$. Now suppose that $I\preceq J$ and that $I\neq J$. If $J\preceq I$, then $I\subseteq J$ and $J\subseteq I$, forcing $I=J$, a contradiction. Hence, in our setting,
\[
I\prec J \text{ if and only if } I\preceq J \text{ and } I\neq J.
\]

\begin{remark}\label{rem:heightHassesum}
    Notice that if we have a chain of ideals $I_0\prec I_1 \prec \dots \prec I_t$ in $\mathcal{I}_0(S)$, then $I_0\subsetneq I_1\subsetneq \dots \subsetneq I_t$, and consequently $t\le \operatorname{g}(S)$ by Proposition~\ref{prop:mins-maxs}~(5). Notice also that if $g_1>\dots>g_{\operatorname{G}(S)}$ are the gaps of $S$, then $(\{0,g_{i+1}\}+S)+(\{0,g_1,\dots,g_{i}\}+S)=\{0,g_1,\dots,g_{i+1}\}+S$. This means that $\{0,g_1,\dots,g_{i}\}+S\prec \{0,g_1,\dots,g_{i+1}\}+S$ for all suitable $i$, which is precisely the chain used in the proof of Proposition~\ref{prop:mins-maxs}~(5), meaning that the largest length of a strictly increasing chain for ideals in $\mathcal{I}_0(S)$ with respect to $\preceq$ has also length $\operatorname{g}(S)+1$.
\end{remark}

With the help of this remark we can prove the following result.

\begin{corollary}\label{cor:isom-clS}
Let $S_1$ and $S_2$ be numerical semigroups. If $(\mathcal{I}_0(S_1),+)$ is isomorphic to $(\mathcal{I}_0(S_2),+)$, then  $\operatorname{g}(S_1)=\operatorname{g}(S_2)$.
\end{corollary}
\begin{proof}
Let $\varphi:\mathcal{I}_0(S_1)\to \mathcal{I}_0(S_2)$ be a monoid isomorphism. For every $I,J\in\mathcal{I}_0(S_1)$, $I\preceq J$ if and only if $\varphi(I)\preceq \varphi(J)$. Thus $\varphi$ maps strictly ascending chains in $\mathcal{I}_0(S_1)$ to strictly ascending chains in $\mathcal{I}_0(S_2)$. By Remark~\ref{rem:heightHassesum}, there is a strictly increasing chain in $\mathcal{I}_0(S_1)$ of length $\operatorname{g}(S_1)+1$, $I_0\prec \dots \prec I_{\operatorname{g}(S_1)}$. Then $\varphi(I_0)\prec \dots \prec \varphi(I_{\operatorname{g}(S)})$ is a maximal strictly ascending chain in $\mathcal{I}_0(S_2)$. In particular, this implies that both $S_1$ and $S_2$ have the same genus. 
\end{proof}

\begin{example}\label{ex:irred-atom-quark-prime}
Let us show what are the irreducible elements, atoms, quarks and primes of $\mathcal{I}_0(S)$, for $S=\langle 5,6,8,9\rangle$. 
The set of irreducible elements of $\mathcal{I}_0(S)$ is
\[
\big\{ \{ 0, 1 \}+S, \{ 0, 2 \}+S, \{ 0, 3 \}+S, \{ 0, 4 \}+S, \{ 0, 7 \}+S, \{ 0, 1, 3 \}+S, \{ 0, 3, 4 \}+S \big\}.
\]
The only atom is $\{0,3,4\}+S$, and the set of quarks is 
\[
\big\{\{ 0, 3, 4 \}+S, \{ 0, 3 \}+S, \{ 0, 4 \}, \{ 0, 7 \}+S\big\}.
\]


Observe that in general, there are more quarks than minimal non-zero ideals (with respect to inclusion). There are no prime elements in $\mathcal{I}_0(S)$.
For instance, let $I=\{0,7\}+S$ and $J=\{0,3,4\}+S$. Then $I+J+J=J+J$, and so $I\preceq J+J$, but $I\not\preceq J$ since $I$ is not included in $J$.

\end{example}

\begin{lemma}
    Let $I$ be $\mathcal{I}_0(S)$. Then $I$ is irreducible if and only if $I \neq J+K$ for any non-zero $J,K\in \mathcal{I}_0(S)\setminus\{I\}$.
\end{lemma}
\begin{proof}
    Suppose that $I$ is irreducible and that $I=J+K$ for some non-zero elements in $\mathcal{I}_0(S)$ with $J\ne I\ne K$. As $I=J+K$ and $J\ne I\ne K$, we have $J\prec I$ and $K\prec I$, contradicting that $I$ is irreducible.

The sufficiency is straightforward.
\end{proof}

The following results ensures that we have at least as many irreducible elements in $\mathcal{I}_0(S)$ as gaps in $S$.

\begin{proposition}\label{prop:irreds-two-el}
Let $S$ be a numerical semigroup. For any gap $g$ of $S$, the ideal $\{0,g\}+S$ is irreducible in $\mathcal{I}_0(S)$.    
\end{proposition}
\begin{proof}
 Suppose by contradiction that this is not the case, and let $I,J\in\mathcal{I}_0(S)\setminus\{S, \{0,g\}+S\}$ be such that $\{0,g\}+S=I+J$. Observe that $I+J=\{0,g\}+S$ forces both $I$ and $J$ to be included in $\{0,g\}+S$. It follows that $g$ is neither in $I$ nor in $J$, and so $g=i+j$ for some non-zero elements $i\in I$ and $j\in J$. But then $i,j\in \{0,g\}+S$. Notice that neither $i$ nor $j$ can be in $S$, since otherwise this would imply that either $\{0,g\}+S\subseteq I$ of $\{0,g\}+S\subseteq J$, contradicting that $I\subsetneq \{0,g\}+S$ and $J\subsetneq \{0,g\}+S$. From $i\in \{0,g\}+S$, we then deduce that $i=g+s$, with $s\in S\setminus\{0\}$, but then $i+j=g+s+j>g$, a contradiction.    
\end{proof}

Notice that in general there are irreducible elements that are not of the form $\{0,g\}+S$ with $g$ a gap of $S$, see Example~\ref{ex:irred-atom-quark-prime} or Figure~\ref{fig:id-4-6-9}.

\begin{figure}
    \centering
    \includegraphics[scale=0.4]{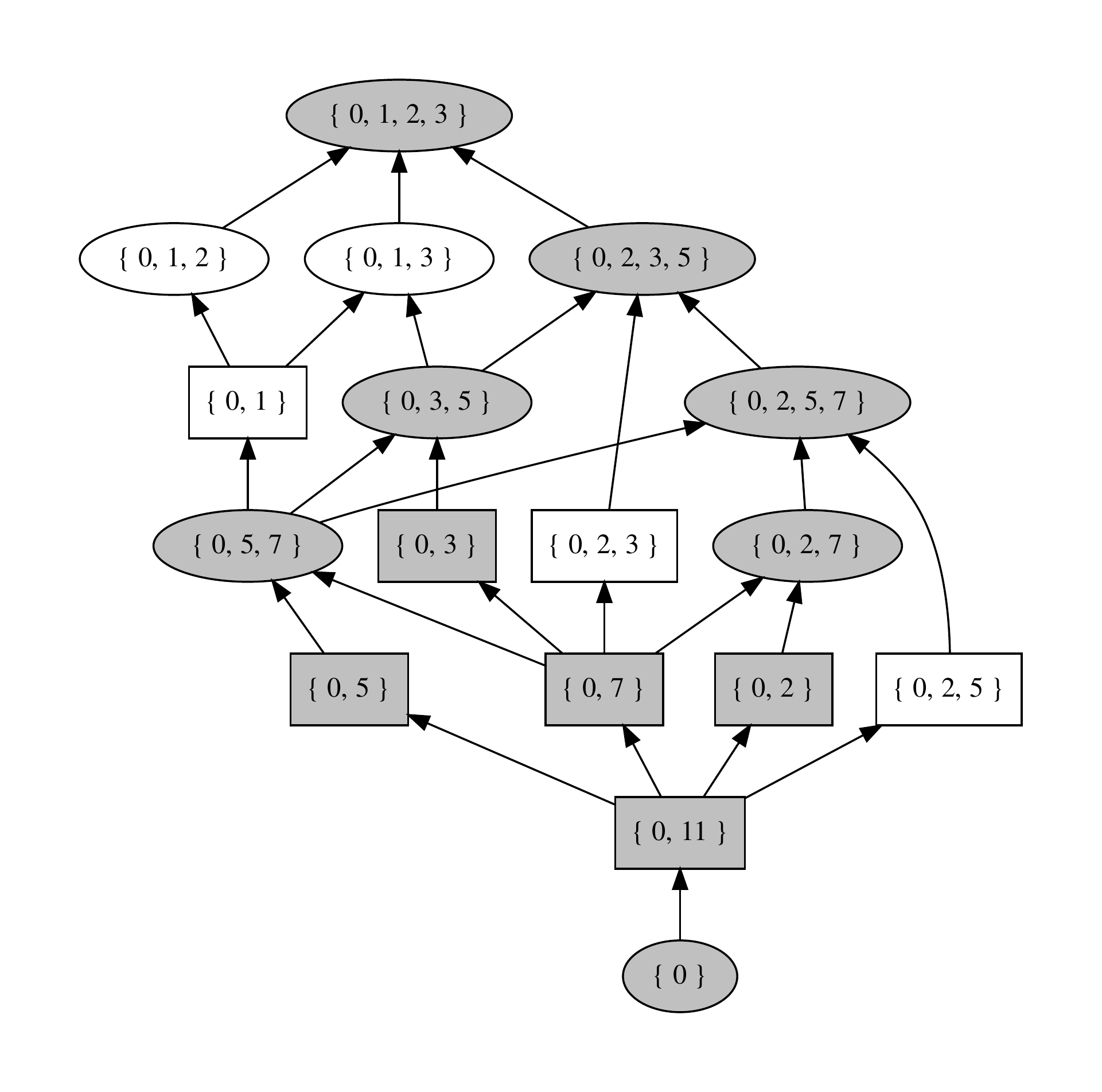}
    \includegraphics[scale=0.4]{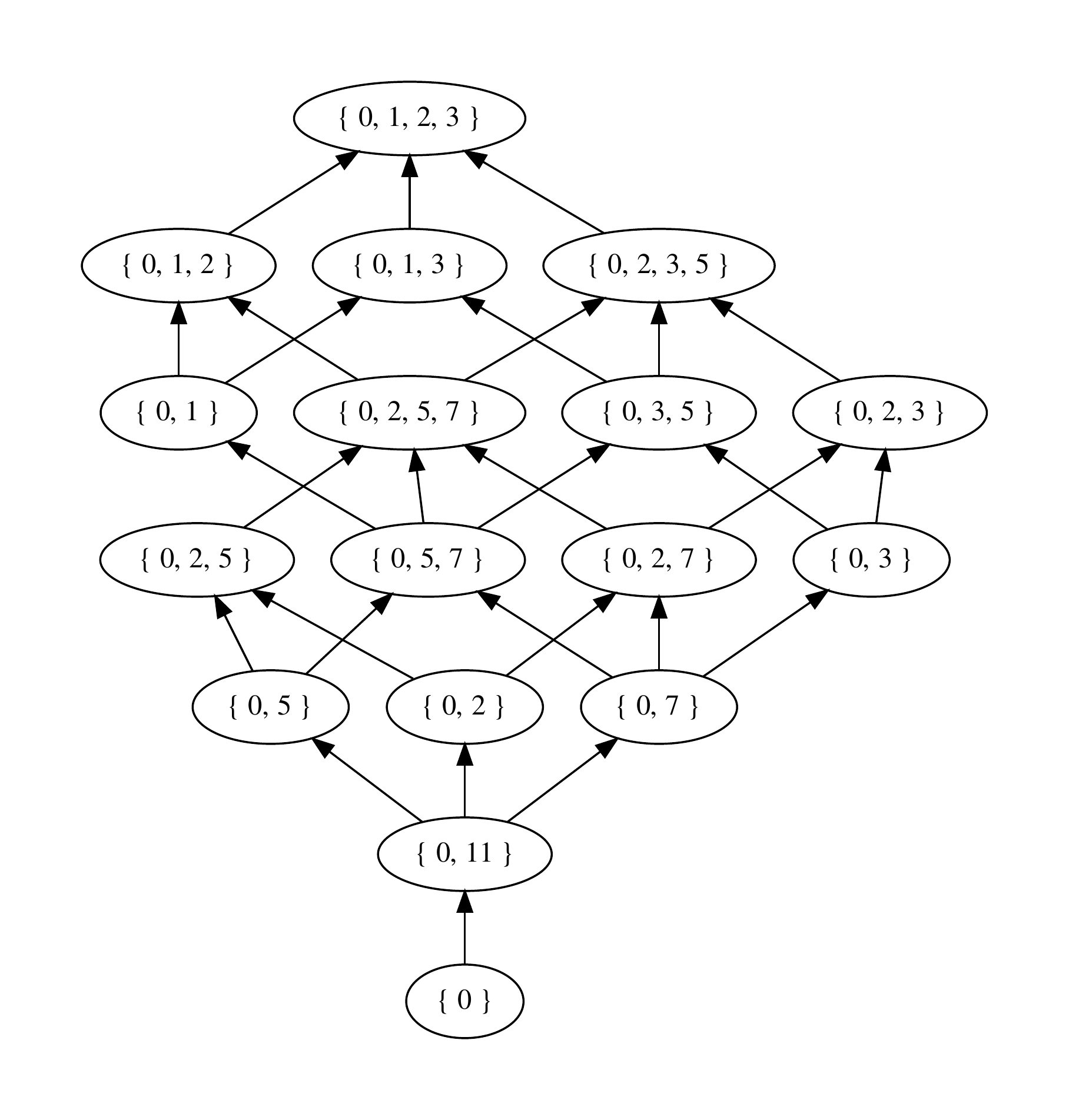}    
\caption{Hasse diagram of $\mathcal{I}_0(\langle 4,6,9\rangle)$ arranged by $\preceq$ (irreducibles in boxes; idempotents filled in gray) and by inclusion}
    \label{fig:id-4-6-9}
\end{figure}

\begin{proposition}
    Let $S$ be a numerical semigroup. Then every ideal of $\mathcal{I}_0(S)$ is a sum of irreducible ideals.
\end{proposition}
\begin{proof}
    The proof follows from the fact there are not infinite strictly descending chains of ideals with respect to $\preceq$.
\end{proof}

If $I\in \mathcal{I}_0(S)$ and $I=I_1+\dots+I_k$ for some $I_1,\dots,I_k$ irreducible elements in $\mathcal{I}_{0}(S)$, then we say that this expression is a factorization (of length $k$) of $I$ into irreducible elements. We say that this factorization is minimal if $I$ cannot be expressed as $\sum_{j\in A} I_j$ with $A\subsetneq \{1,\dots,k\}$.

\begin{example}
    Let $S=\langle 4,6,9\rangle$, and let $I=\{0,1,3\}+S$. Then $I=( \{ 0, 1 \} +S)+ (\{ 0, 3 \}+S )$ is a decomposition of $I$ as sum of irreducible elements in $\mathcal{I}_0(S)$. For $J\in \{ \{0,3\}+S, \{0,5\}+S,\{0,7\}+S,\{0,11\}+S\}$, we have that $I+J=J$. This, in particular, means that $I$ admits factorizations with $n$ irreducible elements for all $n\ge 2$. 
    
    The monoid $\mathcal{I}_0(S)$ does not have the property of having unique minimal factorizations. For instance, $\{0,2,5,7\}+S=(\{0,2,5\}+S)+(\{0,2,5\}+S)=(\{0,2\}+S)+(\{0,5\}+S)$.

    Not all minimal factorizations of an ideal in $\mathcal{I}_0(S)$ need to have the same length. For $S=\langle 5,6,8,9\rangle$, the ideal $I=\{0,2,3,4\}+S$ has minimal factorizations of length two and three. As $I+(\{0,7\}+S)=I$, one can find factorizations of $I$ as sums of $n$ irreducible elements for any integer $n$ greater than one.

    There are examples of non-irreducible ideals with a finite number of factorizations. Take $S=\langle 5, 16, 17, 18, 19\rangle$ and $I= \{0,1,2\}+S$. Then the only factorization of $I$ as sum of irreducible elements is $I=2(\{0,1\}+S)$. This is due to the fact that the only ideal $J$ such that $J+I=I$ is $J=0+S=S$.

\end{example}

\begin{corollary}
    Let $S$ be a numerical semigroup. Then $\mathcal{I}_0(S)$ is cyclic (there exists $I\in \mathcal{I}_0(S)$ such that $\mathcal{I}_0(S)=\{ kI : k\in \mathbb{N}\}$) if and only if $S=\langle 2,3\rangle$.
\end{corollary}
\begin{proof}
Notice that cyclic implies that there is just one irreducible, and this forces $S$ to have just one gap. The other implication is easy to check.
\end{proof}

\begin{example}\label{ex:2-b}
    Let $b$ be an odd integer, and set $S=\langle 2,b\rangle$. Let us see how $\mathcal{I}_0(S)$ looks like. Notice that $\operatorname{G}(S)=\{1,3,\dots,b-2\}$, and that any two gaps are comparable modulo $2$. Thus, $\mathcal{I}_0(S)=\{\{0,g\}+ S : g\in \operatorname{G}(S)\}\cup\{S\}$, and by Proposition~\ref{prop:irreds-two-el}, every non-zero ideal is irreducible. Inclusion is a total ordering in $(\mathcal{I}_0(S),\subseteq)$. Also, notice that if $g_1$ and $g_2$ are gaps of $S$ with $g_1\le g_2$, then $(\{0,g_1\}+S)+(\{0,g_2\}+S)=\{0,g_1\}+S$. This, in particular implies that every non-zero ideal is also a prime ideal. There are no atoms in $\mathcal{I}_0(S)$. The only quark is $\{0,b-2\}+S$.
\end{example}

\begin{lemma}\label{lem:mins-are-quarks}
Let $S$ be a numerical semigroup, and let $I$ be a minimal non-zero ideal in $\mathcal{I}_0(S)$. Then $I$ is a quark.
\end{lemma}
\begin{proof}
    If $I$ is not a quark, then there must be another non-zero $J\in \mathcal{I}_0(S)$ such that $J\prec I$, which would imply that $J\subsetneq I$, a contradiction.
\end{proof}

\begin{example}
    Let $S=\langle  6, 8, 17, 19, 21\rangle$. Then the set of quarks of $\mathcal{I}_0(S)$ is \[\big\{\{ 0, 2, 4, 5 \}+S, \{ 0, 10 \}+S, \{ 0, 11 \}+S, \{ 0, 13 \}+S, \{ 0, 15 \}+S\big\}.\]
    Thus, there are quarks that are not non-zero minimal ideals, and some of them might be generated by elements that are not pseudo-Frobenius numbers.

    In this example, for $Q=\{0,2,4,5\}+S$, the set $\{E\in \mathcal{I}_0(S) : E\subseteq Q\}$ has 25 elements. So $Q$ is far from being minimal.

    For $S=\langle 9,10,\dots, 17\rangle$, we have $\operatorname{t}(S)=8$ (and thus $8$ minimal non-zero ideals with respect to inclusion), while $\mathcal{I}_0(S)$ has 42 quarks.
\end{example}

Next we show that special gaps of the numerical semigroup $S$ are in one-to-one correspondence with idempotent quarks of $\mathcal{I}_0(S)$.

\begin{proposition}\label{prop:idempotent-quarks}
Let $S$ be a numerical semigroup. Then $Q$ is an idempotent quark of $\mathcal{I}_0(S)$ if and only if $Q=\{0,f\}+S$ with $f\in \operatorname{SG}(S)$. 
\end{proposition}
\begin{proof}
Let $Q$ be an idempotent quark. By Remark~\ref{rem:type-mins}, we know that there exists $f\in \operatorname{PF}(S)$ such that $f\in Q$. Notice that $(\{0,f\}+S)+Q=Q$, since for every $s\in S$, $f+s\in Q$ and $Q$ is idempotent. Thus $\{0,f\}+S=Q$, since otherwise $\{0,f\}+S\prec Q$, contradicting that $Q$ is a quark. The fact that $Q+Q=Q$ forces $f+f=2f\in \{0,f\}+S$, which yields $2f\in S$, whence $f\in \operatorname{SG}(S)$.

For the converse, let $f\in \operatorname{SG}(S)$ and set $Q=\{0,f\}+S$. As $f\in \operatorname{PF}(S)$, by Proposition~\ref{prop:mins-maxs}, $Q$ is a minimal non-zero element of $\mathcal{I}_0(S)$ with respect to inclusion. Thus, in light of Lemma~\ref{lem:mins-are-quarks}, $Q$ is a quark. The fact that $Q+Q=Q$ follows easily from the fact that $f\in \operatorname{PF}(S)$ and $2f\in S$.
\end{proof}

Observe that another reading of the last result is that unitary extensions of a numerical semigroup $S$ are precisely the idempotent quarks of $\mathcal{I}_0(S)$. 

We say that a numerical semigroup $T$ is an over-semigroup of $S$ if $S\subseteq T$ \cite{over}. 
\begin{proposition}
    Let $S$ and $T$ be numerical semigroups. Then $T$ is an over-semigroup of $S$ if and only if $T$ is an idempotent in $\mathcal{I}_0(S)$.
\end{proposition}
\begin{proof}
If $T$ is an over-semigroup of $S$, then $T+S=T$ and $T+T=T$, whence $T$ is an idempotent in $\mathcal{I}_0(S)$. Now let $I$ be an idempotent element in $\mathcal{I}_0(S)$. Then $I+I=I$, which means that $I$ is a semigroup, and $S\subseteq I$, which means that (1) the finite complement of $I$ in $\mathbb{N}$ is finite, and (2) it is an over-semigroup of $S$.
\end{proof}

Observe that idempotent ideals are precisely those with reduction number one.

\begin{example}
Of the 17 elements in $\mathcal{I}_0(S)$, with $S=\langle 4,6,9\rangle$, 12 of them are over-semi\-groups (see Figure~\ref{fig:id-4-6-9}).
\end{example}

\begin{remark}
    Notice that if $T$ is an over-semigroup of $S$ and $I\in \mathcal{I}_0(T)$, then $I+S\subseteq I+T\subseteq I$, and thus $I\in \mathcal{I}_0(S)$. Thus, $\mathcal{I}_0(T)\subseteq \mathcal{I}_0(S)$.

    Not every ideal of $\mathcal{I}_0(S)$ is an ideal of an over-semigroup of $S$. This is mainly due to the fact that, there might be more pseudo-Frobenius numbers (which correspond to minimal ideals) than special gaps. For instance, for the semigroup $S=\langle 3,10,17\rangle$, we have $\operatorname{PF}(S)=\{7,14\}$ and $\operatorname{SG}(S)=\{14\}$. In this case the ideal $\{0,7\}+S$ is a minimal non-zero ideal of $S$ (Proposition~\ref{prop:mins-maxs}) that is not an ideal of any over-semigroup of $S$ (if it were so, it would be an ideal of the unique unitary extension of $S$, $S\cup\{14\}$, and $S\cup\{14\}$ is not included in $\{0,7\}+S$).
\end{remark}

\begin{lemma}\label{lem:quark-frob}    
    Let $S$ be a numerical semigroup with Frobenius number $f$. Let $Q$ be a quark of $\mathcal{I}_0(S)$. If $f\in Q$, then $\{0,f\}+S=Q$.
\end{lemma}
\begin{proof}
    We prove that $(\{0,f\}+S)+Q=Q$. The inclusion $E\subseteq (\{0,f\}+S)+Q$ clearly holds. Now take $x\in (\{0,f\}+S)+Q=\{0,f\}+(S+Q)=\{0,f\}+Q$. Then either $x\in Q$ or $x=f+s+q$ for some $s\in S$ and $q\in Q$. If $s+Q>0$, then $f+s+q\in S\subseteq Q$. If $s+q=0$, then $x=f$ which is in $E$ by hypothesis. As $Q$ is a quark and $\{0,f\}+S\preceq Q$, we deduce that $\{0,f\}+S=Q$.
\end{proof}

Next, we see that how is the set of quarks of the ideal class monoid of a symmetric numerical semigroup.

\begin{proposition}\label{prop:char-symm-quarks}
    Let $S$ be a numerical semigroup, $S\neq \mathbb{N}$. Then $S$ is symmetric if and only if $\mathcal{I}_0(S)$ has only one quark. 
\end{proposition}
\begin{proof}
Let $f$ be the Frobenius number of $S$. By Proposition~\ref{prop:mins-maxs} (4) and Lemma~\ref{lem:mins-are-quarks}, we know that $\{0,f\}+S$ is a quark.

Suppose that $S$ is symmetric. Let $Q$ be a quark. By using Remark~\ref{rem:type-mins}, we have that $f\in Q$, since $\operatorname{PF}(S)=\{f\}$. But then Lemma~\ref{lem:quark-frob} ensures that $Q=\{0,f\}+S$.

Now suppose that $\mathcal{I}_0(S)$ has only one quark, which by the first paragraph of this proof must be $\{0,f\}+S$. By Lemma~\ref{lem:mins-are-quarks}, this means that $\mathcal{I}_0(S)$ has at most one minimal element, and Remark~\ref{rem:type-mins} ensures that the type of $S$ is at most one, and thus $S$ is symmetric \cite[Corollary~8]{ns-ap}.
\end{proof}

Notice that if $S$ is symmetric, being $\{0,\operatorname{F}(S)\}+S$ the only quark of its ideal class monoid, $\{0,\operatorname{F}(S)\}+S$ is a prime of $\mathcal{I}_0(S)$. 

\begin{remark}
    Let $S$ be a numerical semigroup with Frobenius number $f$. Then $I\in \mathcal{I}_0(S\cup\{F\})$ if and only if $f\in I$ and $I\in \mathcal{I}_0(S)$. The necessity is easy, since $I+(S\cup\{f\})\subseteq I$ implies that $I+S\subseteq I$ and $f\in I$. For the sufficiency, $I+(S\cup\{f\})\subseteq (I+S)\cup (I+f)\subseteq I\cup (S\cup\{f\})\subseteq I$. 

    In particular, if $S$ is symmetric, by Remark~\ref{rem:type-mins}, we have 
    \[
    \mathcal{I}_0(S)=\{S\}\cup\mathcal{I}_0(S\cup\{f\}).
    \]
    One can observe this in the example depicted in Figure~\ref{fig:id-4-6-9}.

    If $S$ is symmetric with multiplicity $m$, then $f>m$ unless $S=\langle 2,3\rangle$. If $m>2$ and $f<m$, then $S=\langle m,m+1,\dots,m+m-2\rangle$, $f=m-1$. Then $(m-1)-(m-2)=f-(m-2)=1$ must be in $S$ by the symmetry of $S$, forcing $m=1$, a contradiction. The Frobenius number of $\langle 2,n\rangle$, with $n$ odd greater than three, is $n-2$, and so $f>m$ holds also in this case. So, if $S\neq\langle 2,3\rangle$, the only unitary extension of $S$, $S\cup\{f\}$, has also multiplicity $m$. If $(k_1,\dots,k_{m-1})$ and $(k_1',\dots,k_{m-1}')$ are the Kunz coordinates of $S$ and $S\cup\{f\}$, respectively, then $k_i=k_i'$ for $i\neq f \mod m$, and $k_f=k_f'+1$. Thus the upper bound given in Corollary~\ref{cor:bound-ideals-kunz} can be sharpened for symmetric numerical semigroups (other than $\langle 2,3\rangle$),
    \[
    |\mathcal{I}_0(S)|=1+k_f \prod_{\substack{i=1,\\i\neq f\mod m}}^{m-1}(k_i+1).
    \]
\end{remark}

We now proceed with the pseudo-symmetric case.

\begin{proposition}\label{prop:char-pseudo-symm-quarks}
    Let $S$ be a numerical semigroup. Then $S$ is pseudo-symmetric if and only if $\{0,\operatorname{F}(S)\}+S$ and $\{0,\operatorname{F}(S)/2\}+S$ are the only quarks of $\mathcal{I}_0(S)$. 
\end{proposition}
\begin{proof}
Suppose that $S$ is pseudo-symmetric. Let $f$ be the Frobenius number of $S$. By Proposition~\ref{prop:mins-maxs}~(4) and Lemma~\ref{lem:mins-are-quarks}, we know that $\{0,f\}+S$ and $\{0,f/2\}+S$ are quarks.

Suppose that $Q$ is a quark other than $\{0,f\}+S$ and $\{0,f/2\}+S$. By Remark~\ref{rem:type-mins} and the fact that $\operatorname{PF}(S)=\{f/2, f\}$, we have that $f$ or $f/2$ is in $Q$. By Lemma~\ref{lem:quark-frob}, $f$ cannot be in $Q$. Thus $f/2\in Q$ and $\{0,f/2\}+S\subsetneq Q$. Let $g\in Q\setminus(\{0,f/2\}+S)$; which in particular means that $g\neq f/2$. As $S$ is pseudo-symmetric, $f-g\in S$, but then $f\in g+S\subseteq Q$, which is impossible.

Now suppose that the set of quarks of  $\mathcal{I}_0(S)$ is $\{\{0,f/2\}+S,\{0,f\}+S\}$. By Lemma~\ref{lem:mins-are-quarks}, this means that $\mathcal{I}_0(S)$ has at most two minimal elements. Observe that $\{0,f/2\}+S$ and $\{0,f\}+S$ are incomparable with respect to inclusion, and so $\operatorname{Minimals}_\subseteq (\mathcal{I}_0(S)\setminus\{S\})=\{\{0,f/2\}+S,\{0,f\}+S\}$, which by Proposition~\ref{prop:mins-maxs} means that $\operatorname{PF}(S)=\{f/2,f\}$. By \cite[Corollary~9]{ns-ap}, $S$ is pseudo-symmetric.
\end{proof}

With the help of these two propositions we can characterise irreducible numerical semigroups in terms of the number of quarks of its ideal class monoid.

\begin{theorem}
    Let $S$ be a numerical semigroup. Then $S$ is irreducible if and only if $\mathcal{I}_0(S)$ has at most two quarks.
\end{theorem}
\begin{proof}
    If $S=\mathbb{N}$, then $\mathcal{I}_0(S)$ is trivial and has no quarks, and it is indeed the only numerical semigroup for which $\mathcal{I}_0(S)$ has no quarks, since for any other numerical semigroup every non-zero minimal ideal in $\mathcal{I}_0(S)$ is a quark (Lemma~\ref{lem:mins-are-quarks}). Thus, let us suppose that $S\neq \mathbb{N}$.

    \emph{Necessity}. If $S$ is irreducible, then it is either symmetric or pseudo-symmetric, and thus by Propositions~\ref{prop:char-symm-quarks} and \ref{prop:char-pseudo-symm-quarks}, $\mathcal{I}_0(S)$ has at most two quarks.

    \emph{Sufficiency}. Now assume that $\mathcal{I}_0(S)$ has at most two quarks. According to Remark~\ref{rem:type-mins} and Le\-m\-ma~\ref{lem:mins-are-quarks}, we have that the type of $S$ is at most two. If the type is one, then $S$ is symmetric \cite[Corollary~8]{ns-ap}. So it remains to see what happens when the type is two. 
    
    Let $\operatorname{PF}(S)=\{f_1,f_2\}$, with $f_1=\operatorname{F}(S)$. We already know that $\{0,f_1\}+S$ and that $\{0,f_2\}+S$ are non-zero minimal ideals and thus quarks (Proposition~\ref{prop:mins-maxs} and Lemma~\ref{lem:mins-are-quarks}). Let $Q=\{0,f_1-f_2\}+S$. We prove that $Q$ is a quark. As $f_1-f_2\not\in S$ (by definition of pseudo-Frobenius number), Proposition~\ref{prop:irreds-two-el} asserts that $Q$ is irreducible. Suppose that there is some ideal $I\in\mathcal{I}_0(S)\setminus\{S\}$ such that $I\prec Q$. Then there exists $J\in \mathcal{I}_0(S)\setminus\{0\}$ such that $I+J=Q$. But we know that $Q$ is irreducible, and thus $J=Q$. Observe that $f_1\not\in Q$, because if this is the case, then $f_1=f_1-f_2+s$ for some $s\in S$, forcing $f_2=s\in S$, a contradiction. Also, notice that by Remark~\ref{rem:type-mins}, either $f_1\in I$ or $f_2\in I$. As $I\subset Q$, the first case is not possible, and so $f_2\in I$. But then $f_1=f_2+(f_1-f_2)\in I+Q=Q$, a contradiction. Thus $Q$ is a quark, and so either $Q=\{0,f_1\}+S$ or $Q=\{0,f_2\}+S$. Once more, the first case cannot hold, since in particular this would yield $f_1\in Q$, which is impossible. Therefore, $\{0,f_1-f_2\}+S=\{0,f_2\}+S$, which leads to $f_1-f_2=f_2$, or equivalently, $f_2=f_1/2$. By \cite[Corollary~9]{ns-ap}, we deduce that is pseudo-symmetric. 
\end{proof}

Let $S$ be a numerical semigroup, and let $I,J\in \mathcal{I}_0(S)$. Recall that $J$ is said to cover $I$ if $I\prec J$ and there is no $K\in \mathcal{I}_0(S)$ such that $I\prec K\prec J$. We denote by $I^\smile$ the set of ideals in $\mathcal{I}_0(S)$ covering $I$. In this way, the set of quarks is precisely $S^\smile$. Analogously, define $I^\frown$ to be the set of ideals covered by $I$. A natural question (dual to determining $S^\smile$) would be to characterize those ideals belonging to $\mathbb{N}^\frown$. We describe this set in the next result.

\begin{proposition}\label{prop:coveredN}
Let $S$ be a numerical semigroup with multiplicity $m$. Then
\[
\mathbb{N}^\frown=\operatorname{Maximals}_\subseteq(\mathcal{I}_0(S)\setminus\{\mathbb{N}\}).
\]
In particular, $|\mathbb{N}^\frown|=m-1$.
\end{proposition}
\begin{proof}
Let $M$ be in $\operatorname{Maximals}_\subseteq(\mathcal{I}_0(S)\setminus\{\mathbb{N}\})$. Then $\mathbb{N}\subseteq M+\mathbb{N}\subseteq \mathbb{N}$, and thus $M+\mathbb{N}=M$, yielding $M\prec \mathbb{N}$. The maximality of $M$ forces $M\in \mathbb{N}^\frown$.

Now let $I\in \mathbb{N}^\frown$. Let $j=\min(\mathbb{N}\setminus I)$. Notice that $j\le m-1$, since $\{0,\dots,m-1\}\subseteq I$ forces $I=\mathbb{N}$, and we are assuming $I\neq \mathbb{N}$. Set $M=\mathbb{N}\setminus\{j\}$. Then $I\subseteq M$, and from Proposition~\ref{prop:mins-maxs}, we deduce that $M\in \operatorname{Maximals}_\subseteq(\mathcal{I}_0(S)\setminus\{\mathbb{N}\})$. Let us prove that $M=I$. Suppose to the contrary that $I\neq M$, and let $k=\min(M\setminus I)$. Then $j<k$ and $I\prec I+(\{0,k\}+S)$. If $j\in I+(\{0,k\}+S)$, then either $j=x+s$ or $j=x+k+s$, with $x\in I$ and $s\in S$. The equality $j=x+s$ forces $j\in I$, which is impossible, while $j=x+k+s$ yields $j\ge k$, a contradiction. Thus $I+(\{0,k\}+S)\subseteq M$. Consequently $I\prec I+(\{0,k\}+S)\prec \mathbb{N}$, contradicting that $I\in \mathbb{N}^\frown$. 

That the cardinality of $\mathbb{N}^\frown$ is $m-1$ follows, once more, from Proposition~\ref{prop:mins-maxs}.
\end{proof}

\begin{remark}
Let us see what the connections are between the concepts of irreducible, atom, quark and prime in $\mathcal{I}_0(S)$.  
\begin{itemize}
    \item If $I\in \mathcal{I}_0(S)$ is a quark and it is not an idempotent ($I\neq I+I$; the reduction number of $I$ is greater than one), then $I$ is an atom. Notice that $I=J+K$ with $J,K\in \mathcal{I}_0(S)\setminus\{0\}$ forces $J\preceq I$ and $K\preceq I$. As $I$ is a quark, $J=I=K$, contradicting that $I\neq I+I$. 
    \item Let $I\in \mathcal{I}_0(S)$ be an atom, and suppose that there exists $J\in \mathcal{I}_0(S)\setminus\{S\}$ such that $J\prec I$. By definition, there exists $K\in \mathcal{I}_0(S)\setminus\{S\}$ such that $J+K=I$, contradicting the fact that $I$ is an atom. Thus, every atom is a quark.
    \item If $I\in \mathcal{I}_0(S)$ is a quark, then $I$ is an irreducible, since $I=J+K$ with $S\neq J\neq I\neq K\neq S$ forces $J\prec I$, a contradiction. Not every irreducible is a quark (see Example~\ref{ex:irred-atom-quark-prime}).
    \item Notice also that every atom is irreducible. 
    \item In $\mathcal{I}_0(S)$, being a prime implies being irreducible. Assume that $I$ is prime and that $I=J+K$ with $S\neq J\prec I$, $S\neq K\prec I$. Then $J\subsetneq I$ and $K\subsetneq I$. As $I$ is prime and $I\preceq J+K$, either $I\preceq J$ or $I\preceq K$, but this implies that either $I\subseteq J$ or $I\subseteq K$; thus either $I\subseteq J\subsetneq I$ or $I\subseteq K\subsetneq I$, which in both cases is impossible. This proves that $I$ is irreducible. 
\end{itemize}
To summarise, being an atom implies being a quark, which in turn implies being irreducible, and every prime element is irreducible. 
Any other implication between these concepts does not hold in light of Examples~\ref{ex:irred-atom-quark-prime} and \ref{ex:2-b}. These examples also show that there are numerical semigroups for which its ideal class monoid has no primes, and numerical semigroups having ideal class monoid with no atoms.
\end{remark}

\section{Open questions}

We know that if two numerical semigroups have isomorphic ideal class monoids, then they have the same genus (Corollary~\ref{cor:isom-clS}). Proposition~\ref{prop:coveredN} forces their multiplicities to be the same, and Proposition~\ref{prop:idempotent-quarks} tells us that they have the same number of unitary extensions. This leads to the following conjecture. 

\begin{question}
Let $S_1, S_2$ be two numerical semigroups such that $\mathcal{C}\ell(S_1)$ is isomorphic to $\mathcal{C}\ell(S_2)$. Does $S_1=S_2$ hold?
\end{question}

Notice that if $\mathcal{C}\ell(S_1)$ and $\mathcal{C}\ell(S_2)$ are isomorphic, then their respective posets with respect to $\preceq$ are isomorphic. Thus, we could also conjecture something stronger. 

\begin{question}
Let $S_1, S_2$ be two numerical semigroups such that $(\mathcal{C}\ell(S_1),\preceq)$ and $(\mathcal{C}\ell(S_2),\preceq)$ are isomorphic as posets. Does $S_1=S_2$ hold? 
\end{question}

We can also reformulate this question taking inclusion instead of $\preceq$. If the above question has a positive answer, it would also be interesting to determine if there exists an algorithm to recover a semigroup from the Hasse diagram of its ideal class monoid.

In Remark~\ref{width}, we proved that a lower bound for the width of the Hasse diagram of $\mathcal{I}_0(S)$ is $(|\mathcal{I}_0(S)|-2)/(\operatorname{g}(S)-1)$. We also showed that the bound is sharp. However, for numerical semigroups with big ideal class monoids, this bound is not very good: for $S=\langle 5,11,17,18\rangle$, the width of the corresponding Hasse diagram is 25 and $(|\mathcal{I}_0(S)|-2)/(\operatorname{g}(S)-1)=165/11=15$. 

\begin{question}
Are there better bounds for the width of the Hasse diagram with respect to inclusion? What about a possible upper bound?
\end{question}

In general, an element in a monoid might admit different expressions as a sum of irreducible elements. These expressions are known as factorizations. There are a many invariants that measure how far these factorization are from being unique (as happens in unique factorization monoids) or have the same length (half-factorial monoids); see \cite{g-hk} for a detailed description of this theory in the cancellative setting or \cite{t} for a more general scope.

\begin{question}
 Given an ideal $I$ of $\mathcal{I}_0(S)$, can we say something about the number of its factorizations in terms of irreducible elements in $\mathcal{I}_0(S)$? Or even about the lengths of these factorizations? Is this set of lengths an interval? 
\end{question}

Section \ref{quarks} was mainly motivated by previous works, started in \cite{f-t}, on the power monoid $\mathcal{P}_{\mathrm{fin},0}(\mathbb{N})$, the set of subsets of $\mathbb{N}$ containing $0$ and with finitely many elements (and $\mathcal{P}_{\mathrm{fin},0}(S)$ with $S$ a numerical semigroup), endowed with addition $A+B=\{a+b : a\in A, b\in B\}$. Moreover, \cite{b-g} shows the abundance of atoms in $\mathcal{P}_{\mathrm{fin},0}(S)$; however, this is far from being the case for $\mathcal{I}_0(S)$, since we have examples with no atoms at all. In $\mathcal{P}_{\mathrm{fin},0}(S)$, irreducible elements, atoms and quarks are the same \cite[ Proposition~4.11(iii) and Theorem~4.12]{t} (notice that these three concepts, irreducible element, atom and quark, differ in our setting), and so it would make more sense to propose the following question instead. 

\begin{question}
    Let $S$ be a numerical semigroup. What is the ratio between the cardinality of irreducible elements in $\mathcal{I}_0(S)$ and the cardinality of $\mathcal{I}_0(S)$=
\end{question}

We know that we have at least as many irreducibles as the genus of $S$, which in turn is the height of the Hasse diagram of $\mathcal{I}_0(S)$ (minus one).


\section*{Acknowledgements}
The authors would like to thank Salvatore Tringali for helping us understanding the meaning of atoms, irreducibles, quarks and primes in this context.  We would also like to thank Alfred Gerolginger, who highlighted that as a consequence of \cite[Lemma~4.1]{b-g-r}, the monoid (under addition) of integral ideals of a numerical semigroup is unit-cancellative: we give a direct proof in Remark~\ref{rem:non-unit-cancellative}. Alfred Geroldinger also warned us about the potential clash of notation with the notion of class semigroup; Remark~\ref{rem:class-semigroup} was introduced accordingly.

\end{document}